\documentclass[11pt]{amsart}
\usepackage{amssymb}
\usepackage[utf8]{inputenc}

\numberwithin{equation}{section}

\usepackage{amssymb}
\usepackage{amsthm}
\usepackage{graphicx}
\usepackage{mathtools}
\usepackage{amsmath,amsfonts,mathrsfs}  
\usepackage{epsfig}
\usepackage{comment}
\usepackage{color}
\usepackage{amsmath}
\usepackage{amsrefs}
\usepackage{esint}
\usepackage{xcolor}
\usepackage{hyperref}
\usepackage[normalem]{ulem}
\usepackage{bbold}
\usepackage{todonotes}

\def \N {\mathbb{N}}

\def \R {\mathbb{R}}
\def \g {\mathfrak{g}}
\def \F {\mathcal{F}}
\def \E {\mathcal{E}}
\def \H {\mathbb{H}}
\def \G {\mathscr{G}}

\def \B {\mathscr{B}}
\def \eps {\varepsilon}
\def \dd {\mathrm{d}} 

\DeclareMathOperator{\vol}{vol}
\DeclareMathOperator{\loc}{loc}
\DeclareMathOperator{\supp}{supp}
\DeclareMathOperator{\Inv}{Inv}
\DeclareMathOperator{\Reg}{Reg}
\DeclareMathOperator{\Per}{Per}
\DeclareMathOperator{\spn}{span}
\DeclareMathOperator{\Cut}{Cut}
\DeclareMathOperator{\Ad}{Ad}
\DeclareMathOperator{\Bad}{Bad}
\DeclareMathOperator{\Good}{Good}
\DeclareMathOperator{\Radon}{Radon}
\DeclareMathOperator{\FP}{FP}

\newcommand{\uno}{{\mathbb 1}}

\definecolor{cadmiumgreen}{rgb}{0.0, 0.42, 0.24}

\newcommand{\mres}{\mathbin{\vrule height 1.6ex depth 0pt width
0.13ex\vrule height 0.13ex depth 0pt width 1.3ex}}

\newtheorem{definition}[equation]{Definition}
\newtheorem{theorem}[equation]{Theorem}
\newtheorem{corollary}[equation]{Corollary}
\newtheorem{prop}[equation]{Proposition}
\newtheorem{lemma}[equation]{Lemma}
\theoremstyle{definition}
\newtheorem{remark}[equation]{Remark}

\title{Nilpotent groups and biLipschitz embeddings into $L^1$}
\author[Eriksson-Bique]{Sylvester Eriksson-Bique}
\author[Gartland]{Chris Gartland} 
\author[Le Donne]{Enrico Le Donne}
\author[Naples]{Lisa Naples} 
\author[Nicolussi-Golo]{Sebastiano Nicolussi-Golo}
\date{\today}							

\address{Sylvester Eriksson-Bique\\%
Research Unit of Mathematical Sciences \\%
P.O.Box 8000\\%
FI-90014 Oulu\\%
Finland}
\email{\tt sylvester.eriksson-bique@oulu.fi}

\address{Chris Gartland \\ Texas A\&M University, College Station, TX 77843, USA}
\email{\tt cgartland@math.tamu.edu}

\address{Enrico Le Donne \\Department of Mathematics, University of Fribourg, Chemin du Mus\'ee~23, 1700 Fribourg, Switzerland \&  University of Jyv\"askyl\"a, Department of Mathematics and Statistics, P.O. Box (MaD), FI-40014, Finland}
\email{\tt enrico.ledonne@unifr.ch}

\address{Lisa Naples\\%
Macalester College\\%
1600 Grand Avenue\\%
Saint Paul, MN 55105}
\email{\tt lnaples@macalester.edu}

\address{Sebastiano Nicolussi-Golo \\ University of Jyv\"askyl\"a, Department of Mathematics and Statistics, P.O. Box (MaD), FI-40014, Finland}
\email{\tt senicolu@jyu.fi}


\begin{document}
\maketitle

\begin{abstract}
    We prove that if a simply connected nilpotent Lie group quasi-isometrically embeds into an $L^1$ space, then  it is abelian. We reach this conclusion by proving that 
    every Carnot group that biLipschitz embeds into $L^1$ is abelian.
    Our proof follows the work of Cheeger and Kleiner, by considering the pull-back distance of a Lipschitz map into $L^1$ and representing it using a cut measure. We show that such cut measures, and the induced distances, can be blown up and the blown-up cut measure is supported on ``generic'' tangents of the original sets. By repeating such a blow-up procedure, one obtains a cut measure supported on half-spaces. This differentiation result then is used to prove that bi-Lipschitz embeddings can not exist in the non-abelian settings.
\end{abstract}

\tableofcontents
%
%
%
%
%
%
\section{Introduction}
Metric embeddings into the Banach space $L^1([0,1])$ is a well-studied problem \cites{GNRS,KhotNaor,CKL1,NaorPeres2,NaorYoung,BMS} motivated by both pure mathematics \cites{Gromov,Yu} and theoretical computer science \cites{LLR,PR,LeeNaor,AIR}. While many flavors of metric embeddings attract interest (coarse, 1-compression, etc.), the present article is concerned  with quasi-isometric and biLipschitz embeddings; see \cites{Drutu_Kapovich} for these basic definitions.
In the theory of such embeddings into $L^p$ spaces, $L^1$ presents as a unique case. For $p \in (1,\infty)$, the space $L^p$ is uniformly convex and uniformly smooth, and a wealth of tools (random walks \cites{LNP,NaorPeres1}, differentiation \cites{Pansu,CKRNP}) is available to prove obstructions to embeddability. On the other extreme is $L^\infty$, into which every separable metric space isometrically embeds.

The space $L^1([0,1])$ is not uniformly convex (even failing to have the Radon-Nikodym property \cite{Pi}*{Chapter 2}), and thus many examples of metric spaces, such as Gromov hyperbolic groups \cite{Ostrovskii}*{Theorem~1.7(a)}, do biLipschitz embed into it. Nevertheless, there still exist separable metric spaces known not to quasi-isometrically (or even coarsely) embed into $L^1([0,1])$ (e.g., expander graphs \cite{Ostrovskiibook}*{Theorem~4.9}). The main result of this article is to add a large class of geometrically natural examples to this list.

\begin{theorem}\label{thm:mainthm}
A simply connected nilpotent Lie group quasi-isometrically embeds into $L^1$ if and only if it is abelian.
\end{theorem}

As is common in the literature, we have tacitly assumed that the simply connected nilpotent Lie group is equipped with a left-invariant Riemannian distance. The particular choice is irrelevant since all are quasi-isometrically equivalent. Furthermore, in the notation $L^1$ we have intentionally omitted the underlying measure space $[0,1]$ equipped with the Lebesgue measure, and will continue to do so in the sequel. This is also typical in the literature because a separable metric space quasi-isometrically (resp. biLipschitz) embeds into $L^1(\Omega)$ for some measure space $\Omega$ if and only if it quasi-isometrically (resp. biLipschitz) embeds into $L^1([0,1])$, see \cite{Ostrovskiibook}*{Fact 1.20}.

Using a standard asymptotic cone\footnote{An \emph{asymptotic cone} of a metric space $(X,d)$ is an ultralimit with respect to a non-principal ultrafilter, as $j \to \infty$, of the sequence of metric spaces $(X,r_jd)$, where $(r_j)_j$ is any sequence decreasing to 0, see \cite{Drutu_Kapovich}.} argument, the proof of Theorem~\ref{thm:mainthm} may be reduced to proving the seemingly weaker Theorem~\ref{thm:carnotembed} concerning \emph{Carnot groups}, a special class of simply connected nilpotent Lie groups 
equipped with homogeneous subRiemannian distances. Actually, 
  Carnot groups are   exactly the asymptotic cones of nilpotent groups,
see $\S$\ref{sec:prelim} for further background. We include the details of the reduction of the first theorem to the second one after the theorem statement.


\begin{theorem}\label{thm:carnotembed}
Every Carnot group that biLipschitz embeds into $L^1$ is abelian.
\end{theorem}

\begin{proof}[Proof of Theorem~\ref{thm:mainthm} from Theorem~\ref{thm:carnotembed}]
Assume that Theorem~\ref{thm:mainthm} is false, and let $G$ be a nonabelian, simply connected nilpotent Lie group that quasi-isometrically embeds into $L^1$. This induces a biLipschitz embedding from an asymptotic cone of $G$ into an asymptotic cone of $L^1$. By Pansu \cite{Pansu83}, every asymptotic cone of $G$ is a nonabelian Carnot group, and as a corollary of Kakutani's representation theorem, every asymptotic cone of $L^1$ is isometric to another $L^1$ space (\cite{BL}*{Corollary~F.4}). This contradicts Theorem~\ref{thm:carnotembed}.
\end{proof}


Similar general reasoning gives non-embeddability for other classes of groups: every time a group (or a metric space) quasi-isometrically (resp. biLipschitz) embeds into  $L^1$ and one of its asymptotic cones (resp. some of its tangent spaces \'a la Gromov) is a Carnot group, then this Carnot group must be abelian. In particular, in $\S$\ref{sec:moregroups}, we review the case of locally compact groups of polynomial growth and of subRiemannian manifolds.
 
We deduce Theorem~\ref{thm:carnotembed} from another theorem, Theorem~\ref{thm:blowupmetric}, later in this section. Theorem~\ref{thm:carnotembed} should be seen as an extension of a famous result of Cheeger and Kleiner \cite{CKL1}*{Theorem~10.2}, whose work was motivated by the Sparsest Cut Problem and the Goemans-Linial Conjecture (see \cites{LeeNaor,NaorYoung} for detailed discussion). Cheeger-Kleiner's result implies that the simplest nonabelian Carnot group - the Heisenberg group - does not biLipschitz embed into $L^1$, and Theorem~\ref{thm:carnotembed} was even anticipated in their article \cite{CKL1}*{Remark~10.12}. 

We stress that there are unforeseen complications in the proof scheme suggested by \cite{CKL1}*{Remark~10.12}. Indeed, as noted by that remark, \cite{CKL1}*{Theorem~10.2} and its proof should hold for every Carnot group $G$ with the following regularity property: for every finite perimeter subset $E \subset G$ and for $\Per_E$-almost every $x \in G$, \emph{every} tangent of $E$ at $x$ is a half-space (see $\S$\ref{sec:prelim} for background on finite-perimeter sets, their perimeter measures $\Per_E$, blowups, and half-spaces). 
However, this regularity property has proven to be quite elusive, and at the time of this writing, it is unknown whether a general Carnot group posses it or not (importantly, it holds for the Heisenberg group by \cite{FSS03}). The most significant progress made on this finite-perimeter-tangent problem was achieved by Ambrosio-Kleiner-Le Donne in \cite{AKLD}, where it is proved that for every Carnot group $G$, 
every finite-perimeter subset $E \subset G$ has the property that for $\Per_E$-almost every $x \in G$ there is  \emph{some} blowup of $E$ at $x$ that is a half-space. They inferred such a property by proving that {\em iterated} generic blowups of finite-perimeter sets are half-spaces \cite{CKL1}*{Theorem~5.2}.
Because uniqueness of generic blowups  has not been proven, we cannot deduce that the {\em every} blowup is a half-space.
The property that some tangent is a half-space
is not strong enough to directly run the argument from \cite{CKL1}, but with a careful understanding of the methods of Cheeger-Kleiner and the use of the iterated blowups of Ambrosio-Kleiner-Le Donne, pieces from each work can be fit together in just the right way to arrive at the following blowup result. Before specifying the result, we need to introduce blowups metrics: A pseudometric $\rho: G \times G \to [0,\infty)$ is a \emph{blowup metric} at a point $x \in G$ of a pseudometric $d: G \times G \to [0,\infty)$ if there exists a sequence of positive real numbers $(r_j)_j$ decreasing to 0 such that
\[\rho = \lim_{j \to \infty} \frac{1}{r_j}S^*_{x,r_j},\]
where the convergence is locally uniform on $G \times G$ and $\frac{1}{r}S^*_{x,r}d$ is the rescaled\footnote{The operator $\delta_\lambda: G \to G$ denotes the Carnot dilation by factor $\lambda$, see \cite{LD}.} (and translated) metric 
\begin{equation}\label{eq:def_Sstar}
    \frac{1}{r}S^*_{x,r}d(y,z) := \frac{1}{r}d(x\delta_r(y),x\delta_r(z)).
\end{equation}

\begin{theorem} \label{thm:blowupmetric}
For every Carnot group $G$, there exists $k \in \N$ such that for every Lipschitz map $f: G \to L^1$, there exists a $k$-fold iterated blowup metric $\rho$ of the pullback metric\footnote{The \emph{pullback metric} $d_f$ of a map $f: X \to Y$ from a set $X$ into a metric space $(Y,d)$ is the pseudometric on $X$ defined by $d_f(x,y) := d(f(x),f(y))$.} $d_f$ such that $\rho(x,yz) = \rho(x,y)$ for every $x,y \in G$ and every $z \in [G,G]$.
\end{theorem}

\begin{proof}[Proof of Theorem~\ref{thm:carnotembed} from Theorem~\ref{thm:blowupmetric}]
Since every abelian Carnot group is simply the Euclidean $n$-space $\R^n$ as a vector group and as a metric space (for some $n \in \N$), one direction is trivial. For the other direction, suppose  $G$ is a Carnot group admitting a biLipschitz embedding $f: G \to L^1$. Then the pullback metric $d_f$ is biLipschitz equivalent to the Carnot metric $d_G$ on $G$. Since biLipschitz equivalence to $d_G$ is preserved under blowups, the metric $\rho$ given by Theorem~\ref{thm:blowupmetric} is biLipschitz equivalent to $d_G$. Let $z$ be an element in the commutator subgroup $[G,G]$, which we want to prove to be equal to the identity element $1$ of the group $G$. Then we have
\[d_G(1,z) \lesssim \rho(1,z) = \rho(1,1) 
= 0,\]
implying $z=1$. Therefore, the subgroup $[G,G]$ is trivial, meaning $G$ is abelian.
\end{proof}

While the proof of Theorem~\ref{thm:blowupmetric} in full can be found in $\S$\ref{sec:mainproof}, we spend the remainder of this section giving an overview. 

Let $G$ be a Carnot group equipped with a Haar measure
denoted by $\vol_G$ and a Carnot distance $d_G$.
We denote by \emph{$\FP_{\loc}(G)$} the space of equivalence classes of measurable subsets of $G$, called \emph{cuts}, with locally finite perimeter, where two sets are identified if their symmetric difference is $\vol_G$-null. The set $\FP_{\loc}(G)$ inherits a natural Fr\'echet topology as a subset of $L^1_{\loc}(G)$. An \emph{$\FP_{\loc}$ cut measure} on $G$ is a positive Borel measure $\Sigma$ on $\FP_{\loc}(G)$ such that $\int \Per_E(K) \,\dd \Sigma(E) < \infty$ for every compact $K \subset G$. Each such $\Sigma$ gives rise to a \emph{cut metric} $d_\Sigma$ on $G$ satisfying $d_\Sigma(x,y) = \int |\uno_E(x)-\uno_E(y)| \,\dd \Sigma(E)$ for $\vol_G \times \vol_G$-almost every $(x,y) \in G \times G$.

It is proved in \cite{CKL1} that for every Lipschitz map $f: B_G \to L^1$, there exists an $\FP_{\loc}$ cut measure $\Sigma$ on $G$ such that $d_f = d_\Sigma$.
Cheeger and Kleiner's main result, \cite{CKL1}*{Theorem~10.2} (see also \cite{CKL1}*{Remark~10.11}), is that for every $\FP_{\loc}$ cut measure $\Sigma$ on the Heisenberg group $\H$, and for $\vol_\H$-almost every $x \in B_\H$, the rescaled metrics $\frac{1}{r}S^*_{x,r}d_\Sigma$ (defined in \eqref{eq:def_Sstar}) are approximated arbitrarily well, as $r \to 0$, by cut metrics $d_{\hat{\Sigma}}$ with $\hat{\Sigma}$ supported on the collection of half-spaces. This is enough to imply non-biLipshitz embeddability of the Heisenberg group.

In order to understand what happens in a general Carnot group $G$, we make the necessary step of taking a locally uniformly convergent subsequence as $r \to 0$ of the rescaled metrics $\frac{1}{r}S^*_{x,r}d_\Sigma$ 
and study the structure of the resulting limit blowup metric of $d_{\Sigma}$. Before stating our structure result, we introduce new terminology.

\begin{definition}
If $\Sigma$ is an $\FP_{\loc}$ cut measure on $G$ and $\F \subset \FP_{\loc}(G)$, we say that $\F$ \emph{contains the $\Sigma$-generic tangents} if for $\Sigma$-almost-every $E \in \FP_{\loc}(G)$ and $\Per_E$-almost every $x \in G$, every tangent of $E$ at $x$ belongs to~$\F$.
\end{definition}


\begin{theorem} \label{thm:blowupcutmetric}
Let $\Sigma$ be an $\FP_{\loc}$ cut measure on a Carnot group $G$ and $\F \subset \Cut(G)$ 
a collection of cuts such that
\begin{itemize}
    \item $d_\Sigma$ is Lipschitz\footnote{For $\rho$ a $\vol_G \times \vol_G$-measurable pseudometric on $G$, we say that $\rho$ is \emph{Lipschitz} with respect to $d_G$ if there exists $L< \infty$ such that $\rho(x,y) \leq Ld_G(x,y)$ for $\vol_G \times \vol_G$-a.e.~$x,y \in G$. In this case, $\rho$ admits a continuous representative that satisfies $\rho(x,y) \leq Ld_G(x,y)$ for all $x,y \in G$.} with respect to $d_G$,
    \item $\F$ is compact,
    \item $\F$ consists of constant normal cuts,
    \item $\F$ is translation and dilation invariant, and
    \item $\F$ contains the $\Sigma$-generic tangents.
\end{itemize} 
Then, for $\vol_G$-a.e.~$x \in G$, every blowup metric $d_{\Sigma,\infty}$ of $d_\Sigma$ at $x$, and every $R \in (0,\infty)$, there exists a cut measure $\Sigma'$ supported on $\F$ such that $\Sigma'(\F) < \infty$ and $d_{\Sigma,\infty} = d_{\Sigma'}$ on $B_R(0) \times B_R(0)$.
\end{theorem}

The idea for the proof of Theorem~\ref{thm:blowupmetric} is to apply Theorem~\ref{thm:blowupcutmetric} iteratively, obtaining $\FP_{\loc}$ cut measures $\Sigma_k$ such that $d_{\Sigma_k}$ is a $k$-fold iterated blowup of $d_f$ and $\Sigma_k$ is supported on a compact collection of cuts $\F_k$ that contains the $k$-fold (generic) iterated tangents of $\FP_{\loc}(G)$. Crucially, we use intermediate results from \cite{AKLD} to prove that the cuts in $\F_k$ successively have more structure, so that for $k$ large enough, $\F_k$ is the collection of half-spaces (see Lemma~\ref{lem:AKLD}). Theorem~\ref{thm:blowupcutmetric} is proved in $\S$\ref{sec:blowup}, and the formal proof of Theorem~\ref{thm:blowupmetric} can be found in $\S$\ref{sec:mainproof}.

\vspace{.5cm}

\noindent \textbf{Acknowledgements:} The authors thank the organizers of the AMS Mathematical Research Communities program ``Analysis in Metric Spaces'' during which this research was started. The paper benefited from guidance and insights from the organizers of that program, in particular Jeremy Tyson. That program was supported by the National Science Foundation under Grant Number DMS 1641020. S.E.-B. was supported partially by the Finnish Academy grant \# 345005. 
E.L.D. was partially supported by the Academy of Finland (grant 288501 `\emph{Geometry of subRiemannian groups}' and by grant 322898 `\emph{Sub-Riemannian Geometry via Metric-geometry and Lie- group Theory}') and by the European Research Council (ERC Starting Grant 713998 GeoMeG `\emph{Geometry of Metric Groups}').  
S.N.G.~was supported by the Academy of Finland (grant 
322898 
`\emph{Sub-Riemannian Geometry via  Metric-geometry and Lie-group Theory}',
and grant 314172 
`\emph{Quantitative rectifiability in Euclidean and non-Euclidean spaces}').

\section{Preliminaries}\label{sec:prelim}

\subsection{Carnot groups}
We start with a concise introduction to Carnot groups, CC-distances and the theory of perimeter in Carnot groups. 
For more details and additional references see \cite{LD} and \cite{AKLD}.

A \emph{Carnot group} of step $s$ is a connected, simply connected Lie group $G$ whose Lie algebra\footnote{We think of the Lie algebra simultaneously as the tangent space to the identity and as the space of left-invariant vector fields.} $\g$ is stratified, that is, $\g = \oplus_{i=1}^s V_i$ with $[V_1,V_i] = V_{i+1}$ for $i<s$, $[V_1,V_s]=\{0\}$ and $V_s\neq\{0\}$.
We denote by $0$ the identity element of the group $G$, because in the theory of Carnot groups exponential coordinates are often used.
The \emph{rank} of $G$ is $m_1 := \dim(V_1)$ and the topological dimension is $m_\g := \dim(\g)$.
The stratification induces a multiplicative one-parameter group of automorphisms $\delta_\lambda: G\to G$, with $\lambda>0$, called \emph{dilations}, such that $(\delta_\lambda)_*v = \lambda^k v$ for $v\in V_k$.

We call the left-invariant subbundle of $TG$ determined by $V_1$ the \emph{horizontal distribution}.
\emph{Horizontal vector fields} are sections of the horizontal distribution, and we denote by $C^k(G;V_1)$ the space of horizontal vector fields of class $C^k$. 
We equip $V_1$ with a scalar product and thus a norm $v\mapsto |v|$, which extends by left-invariance to the horizontal distribution.
Absolutely continuous curves $\gamma:I\to G$ tangent to $V_1$ have a length defined by integrating the speed $\int_I |\gamma'(t)|\, \dd{} t$.
Infimizing the length of horizontal curves joining two points, we obtain a
 Carnot-Carath\'eodory type  distance  $d_G$ on $G$, which is left-invariant and 1-homogeneous with respect to the dilations $\delta_\lambda$. We call $d_G$ the  \emph{Carnot metric} of $G$. We stress that for every two  choices of scalar product on $V_1$ 
and two choices of  stratification for $\g$
 would yield two distances that are biLipschitz 
 via a Lie group automorphism, see \cite{LD}.

Let $B_r(p)$ denote the open $d_G$-ball of radius $r$ and center $p$ and $B_G:=B_1(0)$ the unit ball centered at the identity. The expression $\vol_G$ denotes the (bi-invariant) Haar measure normalized so that $\vol_G(B_G) = 1$. With this normalization, it happens that, for some $Q \in \N$, $\vol_G(B_r(x)) = r^Q$ for every $x \in G$ and $r > 0$. The value $Q$, which is called the \emph{homogeneous dimension} of $G$, is given by $Q=\sum_{j=1}^s j\dim(V_j)$, and coincides with the Hausdorff dimension of $G$.

\subsection{Cuts and perimeter}\label{subs cuts and perimeter}
The vector space $L_{\loc}^1(G)$ has a natural separable Fr\'echet topology whereby a sequence $(f_j)_j \in L_{\loc}^1(G)$ converges to $f \in L_{\loc}^1(G)$ if and only if $(f_j\big|_{B_R(0)})_j$ converges to $f\big|_{B_R(0)}$ in $L^1(B_R(0))$ for every radius $R < \infty$. This topology is induced by the metric $d(f,g) = \sum_{n=1}^\infty n^{-2}\frac{p_n(f-g)}{1+p_n(f-g)}$, where $p_n$ is the seminorm $p_n(f) = \int_{B_n(0)} |f|\,\dd\vol_G$. 

For $A \subset G$ measurable, we define $\Cut(A)$ to be the set of measurable subsets of $A$ modulo $\vol_G$-null sets, which we call \emph{cuts}. We view $\Cut(A)$ as a closed topological subspace of $L_{\loc}^1(G)$ via $E \mapsto \uno_E$. Indeed, we will often slightly abuse notation by using $E$ instead of $\uno_E$ in our notation.

Following \cite{FSS03}, we define the \emph{perimeter measure} of a cut $E \in \Cut(G)$ as the largest Borel measure which assigns to each open set $\Omega \subset G$ the value 
\[
\Per_E(\Omega) := \sup\left\{\int_E {\rm div } \, \psi \,\dd\vol_G : \psi \in C_c^1(\Omega;V_1),|\psi_p|\leq 1 \ \forall p \in G\right\}.
\]
Here, $C_c^1(\Omega;V_1)$ is the collection of $C^1$ horizontal vector fields with support compactly contained in $\Omega$.
The divergence operator ${\rm div } \, \psi $ is the one induced by the measure $\vol_G$. There are a few equivalent definitions of the perimeter measure and sets of finite perimeter, which in the case of Carnot groups coincide. For example, \cite{CKL1} uses a weak $L^1$-relaxation of the energy of the gradient. See \cite{Ambrosio}*{Example 3.20} for a proof of the equivalence of these definitions. 

If $\Per_E(G) < \infty$, we say that $E$ is of \emph{finite perimeter}. If $\Per_E(B_R(0)) < \infty$ for every $R<\infty$, we say that $E$ is of \emph{locally finite perimeter}. It holds that $E \in \Cut(G)$ is of locally finite perimeter if and only if $\Per_E$ is a Radon measure if and only if, for every horizontal smooth vector field $X$, the distributional derivative $X \uno_E$ is a signed Radon measure. We denote the collection of cuts of locally finite perimeter by $\FP_{\loc}(G)$.

Let $\{X_1, \dots, X_{m_1}\}$ be an orthonormal basis of left-invariant horizontal vector fields. The distributional horizontal gradient $\nabla_H \uno_E := \sum_{i=1}^{m_1} (X_i \uno_E) X_i$ of a locally finite-perimeter cut $E$ is a $V_1$-valued Radon measure which is absolutely continuous with respect to $\Per_E$. The Radon-Nikodym derivative
\begin{equation} \label{def:normal}
    \nu_E := \dfrac{\dd\nabla_H \uno_E}{\dd\Per_E},
\end{equation}
is called the \emph{normal of $E$}. The horizontal gradient and normal do not depend on the choice of orthonormal basis.

Cuts of locally finite perimeter satisfy the following compactness property. The first part has been proven in \cite{GN}*{Theorem~1.28(I)}, and the second part follows from a diagonal argument.
\begin{lemma}[Theorem~1.28(I) of \cite{GN}] \label{lem:BVcompact}
For every $R,C < \infty$ and $x \in G$, the set $\{E \cap B_R(x): \Per_E(B_R(x)) \leq C\}$ is norm-compact in $L^1(B_R(x))$. Consequently, for any function $R \mapsto C_R \in [0,\infty)$, the set $\{E \in \FP_{\loc}(G): \forall R < \infty, \: \Per_E(B_R(0)) \leq C_R\}$ is compact in $L^1_{\loc}(G)$.
\end{lemma}

\subsection{Boundaries}

Following \cite{AKLD}, for $E \in \FP_{\loc}(G)$, we define the \emph{reduced boundary} of $E$, denoted by $\partial^*E$, as the set of all $x \in G$ that  satisfy:
\begin{enumerate}
\item $\Per_E(B_r(x))>0$ for every $r>0$,
\item the limit \[ 
\lim_{r\to 0} \fint_{B_r(x)} \nu_E \,\dd \Per_E
\] 
exists, and
\item 
\[ 
\left| \lim_{r\to 0} \fint_{B(x,r)} \nu_E \,\dd \Per_E \right| = 1.\]
\end{enumerate}

We define the \emph{measure theoretic boundary} of a cut $E \in \Cut(G)$ by
{\everymath={\displaystyle}
\[
\partial_{mt}E := \left\{x \in G: 
\begin{array}{l}
0 < \limsup_{r \to 0} \frac{\vol_G(B_r(x) \cap E)}{r^Q}, \\
\text{ and }  \liminf_{r \to 0} \frac{\vol_G(B_r(x) \cap E)}{r^Q} < 1
\end{array}
 \right\},
 \]
}
and the \emph{strong measure theoretic boundary} by
{\everymath={\displaystyle}
\[
\partial^{str}_{mt}E := \left\{x \in G: 
\begin{array}{l}
0 < \liminf_{r \to 0} \frac{\vol_G(B_r(x) \cap E)}{\vol_G(B_r(0))}, \\
\text{ and }  \limsup_{r \to 0} \frac{\vol_G(B_r(x) \cap E)}{\vol_G(B_r(0))} < 1
\end{array}
 \right\}.
 \]
}
The first two boundaries are already well-documented in the literature. The last boundary has not been studied frequently enough to have gained an accepted name, but it plays an important role for us in Lemma~\ref{lem:strmtblowup}, which is later used in Lemma~\ref{lem:alpha-limit}. For now, we record the fundamental property of these three boundaries.

\begin{lemma}\label{lem:boundaries}
For every Carnot group $G$ and $E \in \FP_{\loc}(G)$, the three boundaries $\partial^*E$, $\partial_{mt}E$, $\partial^{str}_{mt}E$ have full $\Per_E$-measure: $\Per_E(G \setminus \partial^*E) = \Per_E(G \setminus \partial_{mt}E) = \Per_E(G \setminus \partial^{str}_{mt}E) = 0$.
\end{lemma}

\begin{proof}
That $\partial^*E$ has full $\Per_E$-measure is true by results from \cite{Ambrosio}, as observed in \cite{AKLD}*{Theorem~4.16}. That $\partial^{str}_{mt}E$ has full $\Per_E$-measure is the content of \cite{Ambrosio}*{Theorem~4.3(4.2)}. Clearly, $\partial^{str}_{mt}E \subset \partial_{mt}E$, and thus $\partial_{mt}E$ has full $\Per_E$-measure as well.
\end{proof}

\subsection{Tangents of cuts}
When $x \in G$, $r > 0$, we define the map $S_{x,r}: G \to G$ by $S_{x,r}(y) := x\delta_r(y)$. The map $S_{x,r}$ induces pullback operators on $S_{x,r}^*$ on $L^1_{\loc}(G)$ and $L_{\loc}^1(G \times G)$ defined by $S_{x,r}^*(f) = f \circ S_{x,r}$ and $S_{x,r}^*(g) = g \circ (S_{x,r} \times S_{x,r})$, respectively. When $E \in \Cut(G)$ (and hence identified with $\uno_E \in L^1_{\loc}(G)$), the formula reads $S_{x,r}^*(E) = \delta_{1/r}(x^{-1}E)$. We also define $\delta^x_r: \Cut(G) \to \Cut(G)$ by $\delta^x_r(E) := xS_{x,r}^*(E)$.

A \emph{tangent}, or {\em blowup}, of $E \in \Cut(G)$ at $x$ is the $L^1_{\loc}(G)$-limit of the sequence of cuts   $\delta^x_{r_j}(E)$ for some sequence $r_j$ decreasing to 0.

The following lemma is essentially present in \cite{FSS03}, but we include our own statement and proof for clarity. The proof is an application of a density estimate of Ambrosio \cite{Ambrosio} and Lemma~\ref{lem:BVcompact}.

\begin{lemma}
\label{lem:precompact}
Let $E \in \FP_{\loc}(G)$. Then for $\Per_E$-a.e.~$x \in G$, the family $\{S_{x,r}^*(E)\}_{r \in (0,1]}$ is precompact in $L_{\loc}^1(G)$.
\end{lemma}

\begin{proof} 
First, we show that for a given a function $f \in L^1_{\loc}(G)$, the map $S_f: G \times (0,\infty) \to L^1_{\loc}(G)$ defined by $S_f(x,r) = S_{x,r}^*(f)$ is continuous. Indeed, if $f$ is a continuous function, this follows since $S_f(x_j,r_j)$ converges locally uniformly to $S_f(x,r)$ whenever $(x_j,r_j)_j$ converges to $(x,r)$. On the other hand, if $(f_i)_i$ is a convergent sequence of $L^1_{\loc}(G)$ functions, then a change of variables argument shows that $S_{f_i}\to S_{f}$ converges uniformly on compact subsets. Thus, by density of continuous functions, $S_f$ is continuous for all $f\in L^1_{\loc}(G)$. This proves, for all $r_0 \in (0,1]$,  $x \in G$, and $E \in \FP_{\loc}(G)$, that $\{S_{x,r}^*(E)\}_{r \in [r_0,1]}$ is compact in $L_{\loc}^1(G)$ since it is the continuous image of the compact set $\{x\} \times [r_0,1]$. Consequently, it remains to prove that, for $\Per_E$-a.e.~$x \in G$ and for every sequence $r_j$ decreasing to 0, $\{S_{x,r_j}^*(E)\}_{j=1}^\infty$ has a convergent subsequence. 

Let $R< \infty$. A diagonal argument reduces the problem to showing that $\{S_{x,r_j}^*(E)\}_{j=1}^\infty$ is norm-precompact in $L^1(B_R(0))$. We shall accomplish this by showing that, for $\Per_E$-a.e.~$x \in G$, there exist $C < \infty$ and $\rho > 0$ such that for every $r < \rho$, $\Per_{S_{x,r}^*(E)}(B_R(0)) \leq C$. Then Lemma~\ref{lem:BVcompact} immediately implies the conclusion.

Let $Q$ be the Hausdorff dimension of $G$. For any $x \in G$ and $r > 0$, by left-invariance and $(Q-1)$-homogeneity of perimeter (see, e.g., \cite{FSS03}*{Remark~2.20}), we have $\Per_{S_{x,r}^*(E)}(B_R(0)) = r^{1-Q}\Per_E(B_{Rr}(x))$. By \cite{Ambrosio}*{Theorem~4.3}, for $\Per_E$-a.e.~$x \in \partial^*E$, there exist $C < \infty$ and $\rho > 0$ such that for every $r < \rho$, $\Per_E(B_{Rr}(x)) \leq Cr^{Q-1}$. Combining these two yields the desired inequality.
\end{proof}

An important property of boundaries we require is that membership of a point $x$ in the boundary of a cut $E \in \FP_{\loc}(G)$ is preserved under blowups at $x$ for $\Per_E$-almost every $x$. It is not clear that the reduced boundary satisfies this property, but it is clear for the strong measure-theoretic boundary, and this is the reason for its appearance in the article.

\begin{lemma} \label{lem:strmtblowup}
For every Carnot group $G$, $E \in \Cut(G)$, $x \in \partial^{str}_{mt}E$, and any tangent $F$ of $E$ at $x$, $x \in \partial^{str}_{mt}F$.
\end{lemma}

\begin{proof}
Let $x \in G$ and $(r_j)_j$ a sequence decreasing to 0 such that $F$ is the $L^1_{\loc}(G)$ limit of $\delta^x_{r_j}(E)$. Then, for any $R>0$, we have
\begin{align*}
    \dfrac{\vol_G(B_R(x) \cap F)}{R^Q} &= \lim_{j \to \infty} R^{-Q}\vol_G(B_R(x) \cap x\delta_{r_j^{-1}}(x^{-1}E)) \\
    &= \lim_{j \to \infty} (Rr_j)^{-Q}\vol_G(B_{Rr_j}(x) \cap E) \\
    &\geq \liminf_{r \to 0} \dfrac{\vol_G(B_r(x) \cap E)}{r^Q}.
\end{align*}
By changing the last line in this argument, we also get
\begin{align*}
    \dfrac{\vol_G(B_R(x) \cap F)}{R^Q} \leq \limsup_{r \to 0} \dfrac{\vol_G(B_r(x) \cap E)}{r^Q}.
\end{align*}
Consequently, $x \in \partial^{str}_{mt}F$ whenever $x \in \partial^{str}_{mt}E$.
\end{proof}

\section{Cuts of constant normal}

For a locally finite-perimeter cut $E$ in a Carnot group $G$, we say that $E$ has \emph{constant normal} if $\nu_E$ equals a constant vector $\vol_G$-almost everywhere. This is equivalent to the existence of a horizontal vector $\nu \in V_1$ such that $\nu \uno_E \geq 0$ (as a measure) and for each $\xi \in V_1$ with $\xi \perp \nu$, the equation $\xi \uno_E = 0$ holds. In this case, $\nu_E \equiv \nu$ almost everywhere. For finite-perimeter cuts $E$ with constant normal $\nu$, the perimeter measure can be represented by the distributional derivative $\Per_E = \nu \uno_E$. For more details, see \cites{FSS03,AKLD,BLD}. The following lemma collects important facts about constant-normal cuts from \cite{BLD}.

\begin{lemma} \label{lem:constantnormalperimeter}
Let $Q$ be the Hausdorff dimension of $G$. Then there exist constants $0 < c \leq C < \infty$ (depending only on $G$) such that, for any $0 < R < \infty$, constant normal cut $E \in \Cut(G)$, and $x \in \partial^*E$,
\begin{equation*}
cR^{Q-1} \leq \Per_E(B_R(x)) \leq CR^{Q-1}.
\end{equation*}
Furthermore, $\partial^* E = \partial_{mt}E = \partial^{str}_{mt}E = \supp(\Per_E)$.
\end{lemma}

\begin{proof}
The estimates for $\Per_E(B_R(x))$ follows from \cite{BLD}*{Proposition~3.11}. The equality $\partial^* E = \partial_{mt}E = \supp(E)$ is stated in \cite{BLD}*{Proposition~3.7(3)}. The containment $\partial^{str}_{mt}E \subset \partial_{mt}E$ is true for any cut $E$, and the other containment for constant normal cuts follows from \cite{BLD}*{Theorem~1.1 and Proposition~3.8}.
\end{proof}




As we blow up sets of constant normal, at almost every point, they become more regular. To measure this regularity, following \cite{AKLD}, we consider the span of certain invariant directions.

\begin{definition}\label{def:mathcalF}
For a given cut $E \in \Cut(G)$, let $\Inv(E) := \{X \in \g: X\uno_E = 0\}$ denote the $E$-{\em invariant directions}, $\Inv_0(E) := \cup_{i=1}^s (V_i \cap \Inv(E))$ the {\em homogeneous $E$-invariant directions}, and $\Reg(E) := \{X \in \g: X\uno_E$ is a Radon measure$\}$ the $E$-{\em regular directions}.

Let $\F_k$ denote the collection of cuts $E \in \Cut(G)$ such that

\begin{enumerate}
    \item $E$ has constant normal.
    \item $\dim(\spn(\Inv_0(E))) \geq k$.
\end{enumerate}
\end{definition}
A cut $E$ with constant normal is said to be a \emph{half-space} if $\Inv_0(E)\cap \bigcup_{i=2}^s V_i = \bigcup_{i=2}^s V_i$. Note that $\F_{m_\g-1}$ is precisely the collection of half-spaces (recall that $m_\g$ is the   dimension of $\g$).

In the following, the set $C_c^\infty(G)$ is the space of compactly supported smooth functions, called {\em test functions}. It is equipped with the usual topology on test functions coming from the direct limit topology of $C_c^\infty(B_n(0))$ with $n\to\infty$. The dual space $C_c^\infty(G)^*$ is called the {\em space of distributions} on $G$.

\begin{lemma}\label{lem:distribution}
If $X_j \to X \in \g$ and $u_j,u \in L^\infty(G)$ have $L^\infty$-norms bounded by 1 with $u_j \to u$ $\vol_G$-a.e., then the sequence $(X_j u_j)_j$ converges to $Xu$ in $C_c^\infty(G)^*$.
Moreover, if $X_ju_j$ are positive Radon measures, then $Xu$ is also a positive Radon measure and $X_ju_j$ weak* converges to $Xu$.
\end{lemma}
\begin{proof}
Fix a test function $\phi \in C^\infty_c(G)$.
As vector fields, $X_j\to X$ uniformly on compact sets.
Since $\phi$ is smooth and has compact support, $X_j\phi\to X\phi$ uniformly.
Denote by $K$ the support of $\phi$ and by $M=\sup_j \| u_jX_j\phi \|_{L^\infty(G)}$.
Then $| u_j X_j\phi| \le M 1_K$ for all $j$, where $\|M1_K\|_{L^1} = M\vol_G(K)<\infty$.
By the Lebesgue's Dominated Convergence Theorem, we conclude that
\begin{align*}
\lim_{j\to\infty} \langle X_j u_j,\phi \rangle
&= \lim_{j\to\infty} -\langle u_j, X_j\phi \rangle 
= \lim_{j\to\infty} -\int u_jX_j\phi \: \,\dd\vol_G \\
&= -\int uX\phi \: \,\dd\vol_G 
= -\langle u, X\phi \rangle 
= \langle Xu,\phi \rangle .
\end{align*}
Since $\phi\in C_c^\infty(G)$ is arbitrary, we obtain $X_ju_j\to Xu$ in $C_c^\infty(G)^*$.

The last part of the statement is a 
standard classical fact, see \cite{hormander}*{Theorem~2.1.90}.
\end{proof}

\begin{lemma}\label{lem:weak*conv of horizonal gradients}
    If $E_j$ is a sequence of constant-normal cuts converging to $E$ in $L_{\loc}^1(G)$, then $E$ has constant normal and $\Per_{E_j}\to \Per_{E}$ weak*.
\end{lemma}
\begin{proof}
    Let $(E_{j_\ell})_\ell$ be an arbitrary subsequence of $(E_j)_j$. For every $\ell$, fix an orthonormal basis $X^{j\ell}_1,\dots X^{j_\ell}_r$ of $V_1$ so that $E_{j_\ell}$ has constant normal $X^{j_\ell}_1$, $X^{j_\ell}_1 \uno_{E_{j_\ell}}\ge0$, and $X^{j_\ell}_i \uno_{E_{j_\ell}}=0$ for $i>1$. By passing to subsequences, we may assume that $E_{j_\ell}$ converges to $E$ almost everywhere and that the vectors $X^{j_\ell}_i$ converge to some $X^{\infty}_i$ for all $i$. It follows that $X^\infty_1, \dots X^\infty_r$ forms an orthonormal basis of $V_1$. By Lemma~\ref{lem:distribution} we have
    \begin{align*}
    &X^{\infty}_1 \uno_E = \lim_{\ell \to \infty} X^{j_\ell}_1 \uno_{E_{j_\ell}} \geq 0, \\
    &X^\infty_i \uno_E = \lim_{\ell \to \infty} X_i^{j_\ell} \uno_{E_{j_\ell}} = 0,\text{ for }i>1 ,
    \end{align*}
    and hence $E$ has constant normal $\nu = X^\infty_1$.
    
    By Lemma~\ref{lem:distribution} again, and by the distributional derivative representation of perimeter measures for constant normal cuts, we have
    \[
    \lim_{\ell \to \infty} \Per_{E_{j_\ell}} = \lim_{\ell \to \infty} X^{j_\ell}_1 \uno_{E_{j_\ell}} = X^\infty_1 \uno_E = \Per_E,
    \]
    where the convergence is weak*. Thus, since every subsequence of $\Per_{E_j}$ has a subsequence that weak* converges to $\Per_E$, the full sequence $\Per_{E_j}$ weak* converges to $\Per_E$.
\end{proof}


\begin{theorem} \label{thm:L1closed}
For each $k\in \mathbb{N}$, the collection $\F_k$ from Definition~\ref{def:mathcalF} is closed in $L_{\loc}^1(G)$. 
\end{theorem}
\begin{proof} Fix $k\in \N$. Let $E_j \in \F_k$ be a sequence of cuts and $E \in \Cut(G)$ such that $\uno_{E_j} \to \uno_{E} \in L_{\loc}^1(G)$. By Lemma~\ref{lem:weak*conv of horizonal gradients}, $E$ has constant normal.

Now, for each $j$, let $W_j := \spn(\Inv_0(E_j))$ and $k_j := \dim(W_j) \geq k$. Then we can write $W_j = W_j^1 \oplus \dots W_j^s$ with $W_j^i := V_i \cap \Inv(E)$. Let $k_j^i := \dim(W_j^i)$ so that $k_j = k_j^1 + \dots k_j^s$. By passing to subsequences, we may assume that there is some integer $k^i$ such that $k_j^i = k^i$ for all $j \in \N$. Since the Grassmannian $\mathrm{Gr}(k^i,V_i)$ is compact, we may assume that there exists a $k^i$-dimensional subspace $W^i$ such that $W_j^i \to W^i \in Gr(k^i,V_i)$. Let $X \in W^i$. Then there exists a sequence $X_j \in W_j^i$ such that $X_j \to X$. Then by Lemma~\ref{lem:distribution}, we have
\begin{align*}
    X\uno_E = \lim_{j \to \infty} X_j\uno_{E_j} = 0,
\end{align*}
showing $W^i \subset \Inv_0(E)$ for every $1 \leq i \leq s$. Hence,
\begin{align*}
    \dim(\spn(\Inv_0(E))) \geq \dim(W^1) + \dots \dim(W^s) = k^1 + \dots k^s \geq k.
\end{align*}
By definition, this implies $E \in \F_k$.
\end{proof}

In the next theorem we collect a few crucial properties of constant normal cuts.
\begin{theorem} \label{thm:Fkcompact}
If $\F \subset \FP_{\loc}(G)$ is a closed collection of constant-normal cuts, then $\F$ is compact. Further, for each compact set $A \subset G$, the collection $\F^A =\{E \in \F : \partial^* E \cap A \neq \emptyset \}$ is compact.

In particular, for any $k\in \N$ the set $\F_k$ is compact in $L_{\loc}^1(G)$.

\end{theorem}
\begin{proof}
If $\F \subset \FP_{\loc}(G)$ is any closed collection of cuts, then Lemmas~\ref{lem:BVcompact} and~\ref{lem:constantnormalperimeter} yield compactness. 

We next prove the second claim. Since $\F$ is compact, it suffices to prove that $\F^A$ is closed. 
Take an arbitrary sequence $\{E_j\}_{j\in \N}$ in $\F^A$ converging to some $E\in \F$. By definition, there exists a sequence $\{x_j\}_{j\in \N}$ of points with $x_j\in A\cap \partial^* E_j$ for $j\in \N$. Then, since $A$ is compact, there exists a subsequence that converges to some $x \in A$.
To show that $E\in \F^A$, it suffices to prove that $x\in \partial^* E$. 
By Lemma~\ref{lem:weak*conv of horizonal gradients}, the  perimeter measures $\Per_{E_j}$ converge weak* to $\Per_{E}$. 
By Lemma~\ref{lem:constantnormalperimeter} there is a constant $c>0$ so that $\Per_{E_n}(B_{r/2}(x_n)) \geq c r^{Q-1}$ for all $n\in \N$ and all $r>0$. 
By weak* convergence, $\Per_E(B_r(x)) \geq c r^{Q-1}$ for each $r>0$. Therefore $x\in \supp(\Per_E)$ and by Lemma~\ref{lem:constantnormalperimeter} again, we have $x\in \partial^* E$.

For the final claim, take any $k\in \N$. By Theorem~\ref{thm:L1closed}, the collection $\F_k$ is a closed set of cuts. Therefore, by the first claim $\F_k$ is also compact, as claimed. 
\end{proof}





\section{Modified Cheeger-Kleiner}

The goal of this section is to prove an infinitesimal regularity result in the spirit of Cheeger-Kleiner's result  \cite{CKL1}*{Theorem~10.2}, see Remark~\ref{dicember_in_Duedingen}, adapted to our more general setting. Roughly, the result says that the blowup of a cut metric is a cut metric on blowups.
To obtain the result we modify \cite{CKL1}*{Sections 6-10}. For convenience, we name our subsections according to the corresponding sections in \cite{CKL1}.  We include proofs of lemmas when significant modifications are made or when essential for clarity of exposition; otherwise, we refer the reader to proofs in \cite{CKL1}.  

Given a positive Borel measure $\Sigma$ on $\Cut(G)$, we say that $\Sigma$ is a \emph{cut measure} if it satisfies $\int_{\Cut(G)} \vol_G(E \cap B_R(0)) \,\dd \Sigma(E) < \infty$ for every $R < \infty$. Note that $\Sigma(\Cut(G)) < \infty$ is sufficient for $\Sigma$ to be a cut measure, and we will frequently use this fact without mention. 
We also adopt the convention that $\Sigma(\{\emptyset\})=0$ for every cut measure $\Sigma$. Notice that this ensures that  $\Sigma$ is $\sigma$-finite: Indeed, if we define $\Omega_n := \{E \in \Cut(G): \vol_G(E \cap B_n(0)) > \frac{1}{n}\}$, then $\Sigma(\Cut(G) \setminus \cup_n \Omega_n) = \Sigma(\{\emptyset\}) = 0$ and $\frac{1}{n}\Sigma(\Omega_n) \leq \int_{\Omega_n} \vol_G(E \cap B_n(0)) \,\dd \Sigma(E) < \infty$.

The importance of cut measures is that they give rise to \emph{cut metrics}. By \cite{CKmonotone}*{page 345}, for every cut measure $\Sigma$, there exists a $\Sigma \times \vol_G \times \vol_G$-measurable function $(E,x,y) \mapsto d_E(x,y): \Cut(G) \times G \times G \to [0,\infty)$ called an \emph{elementary cut metric}, such that, for $\Sigma$-a.e.~$E \in \Cut(G)$, we have $d_E(x,y) = |\uno_E(x)-\uno_E(y)|$ for $\vol_G \times \vol_G$-a.e.~$(x,y) \in G \times G$. The cut metric $d_\Sigma$ is defined by
\begin{equation} \label{eq:cutmetricdef}
d_\Sigma(x,y) := \int_{\Cut(G)} d_E(x,y) \,\dd \Sigma(E).
\end{equation}
Equation \eqref{eq:cutmetricdef} gives a well-defined element of $L^1_{\loc}(G \times G)$ by Fubini Theorem (which is applicable since both $\Sigma$ and $\vol_G$ are $\sigma$-finite) and uniquely determines $d_\Sigma$ up to $\vol_G \times \vol_G$-null sets. In particular, \eqref{eq:cutmetricdef} holds pointwise $\vol_G \times \vol_G$-almost everywhere, but in general not everywhere. Whenever $d_\Sigma$ is Lipschitz with respect to $d_G$, we choose the continuous representative for $d_\Sigma$.

Our differentiation and regularity results concern cut measures supported on more regular collections of cuts. A cut measure $\Sigma$ on $G$ is an \emph{$\FP_{\loc}$ cut measure} if it is supported on cuts of locally finite perimeter and, for every $R< \infty$,
\begin{align*}
    \int_{\Cut(G)} \Per_E(B_R(0)) \: \,\dd \Sigma(E) < \infty.
\end{align*}


The crucial fact we need is the following.


\begin{prop}\label{prop:cut-meas-existence}
Let $G$ be a Carnot group. For every Lipschitz map $f:G \to L^1$, there exists an $\FP_{\loc}$ cut measure $\Sigma$ on $G$ such that
\[d_f(x,y) = d_\Sigma(x,y), \text{ for } \vol_G \times \vol_G\text{-a.e.~} (x,y) \in G \times G,
\]
where $d_f(x,y) = \|f(x)-f(y)\|_{L^1}$.
In particular, $d_\Sigma$ has a continuous representative and $\Sigma$ is supported on locally finite-perimeter cuts.
\end{prop}

\begin{remark}
It is not stated in \cites{CKL1,CKmonotone} that the cut measure $\Sigma$ satisfying $d_f = d_\Sigma$ obeys the convention $\Sigma(\{\emptyset\}) = 0$. However, it is obvious that the value of $\Sigma(\{\emptyset\})$ has no effect on $d_\Sigma$. Since our only use of cut measures in this article is through their induced cut metrics, we always redefine $\Sigma(\{\emptyset\})$ to be 0 whenever $\Sigma$ would otherwise be a cut measure.
\end{remark}

\begin{proof}[Proof of Proposition~\ref{prop:cut-meas-existence}]
The existence of a cut measure $\Sigma$ satisfying $d_f = d_\Sigma$ almost everywhere is stated in \cite{CKmonotone}*{Theorem~2.9} (and that article cites \cite{CKL1}*{Proposition~3.40} for the proof). Since $f$ is Lipschitz, it follows from \cite{CKL1}*{Proposition~4.17} that $\Sigma$ is an $\FP_{\loc}$ cut measure.
\end{proof}

\begin{remark}
    \label{dicember_in_Duedingen}
For $G $ equal to the Heisenberg group  $ \H$, Cheeger-Kleiner proved in \cite{CKL1}*{Theorem~10.2} that for any $\FP_{\loc}$ cut measure $\Sigma$ and $\vol_\H$-a.e.~$x \in B_\H$,
\begin{align*}
    \lim_{r \to 0} \inf_{\bar{\Sigma} \in HS} \|\tfrac{1}{r}S_{x,r}^*(d_\Sigma) - d_{\bar{\Sigma}}\|_{L^1(B_\H \times B_\H)} = 0,
\end{align*}
where $HS$ is the collection of cut measures supported on half-spaces.  The choice of $\bar{\Sigma}\in HS$ is appropriate when $G = \H$ since by \cite{FSS03}*{Theorem~3.1}, the blowup of every locally finite-perimeter set at a point in reduced boundary is a unique half-space.
\end{remark}
In the setting of general Carnot groups, it remains unknown
whether all generic tangents are half-spaces
 or, equivalently because of \cite{AKLD},
whether there is a unique tangent. Thus, we replace $HS$ with a collection $\F$ of constant normal cuts satisfying properties specified below.




\begin{definition}
Let $\F \subset \Cut(G)$ be a collection of cuts. 
\begin{enumerate}
    \item For a given locally finite-perimeter cut $E \in \FP_{\loc}(G)$, we say that $\F$ \emph{contains the generic tangents of $E$} if for $\Per_E$-a.e.~$x \in G$, every tangent of $E$ at $x$ belongs to $\F$. 
    \item For a given $\FP_{\loc}$ cut measure $\Sigma$ on $G$, we say that $\F$ \emph{contains the $\Sigma$-generic tangents} if for $\Sigma$-a.e.~$E \in \Cut(G)$, the family $\F$ contains the generic tangents of $E$.
\end{enumerate}
\end{definition}

In terms of this definition, we can express the result of \cite{FSS03} for $G=\H$ and $\F=HS$. In \cite{FSS03} it was shown that for every cut of locally finite perimeter $E \in \FP_{\loc}$, the family $\F$ contains the generic tangents of $E$. Therefore, if $\Sigma$ is an $\FP_{\loc}$ cut measure, then $\F$ contains the $\Sigma$-generic tangents.

We emphasize that we do {\emph not} require $\F$ to contain every tangent at every $x\in E$ for every $E\in\Cut(G)$, but rather allow for flexibility up to sets of $\Per_E$- and $\Sigma$-measure zero.

\begin{theorem} \label{thm:modCK}
Let $\Sigma$ be a locally $\FP_{\loc}$ cut measure on a Carnot group $G$ and let $\F \subset \Cut(G)$ be a collection of cuts such that
\begin{itemize}
    \item $\F$ is compact,
    \item $\F$ consists of constant normal cuts,
    \item $\F$ is translation and dilation invariant, and
    \item $\F$ contains the $\Sigma$-generic tangents.
\end{itemize}
Then, for $\vol_G$-a.e.~$x \in G$, there exists a constant $K_x \in (0,\infty)$ such that 
\begin{equation} \label{eq:modCK}
     \lim_{r \to 0} \inf_{\bar{\Sigma} \in \mathscr{F}(K_x)} \|\tfrac{1}{r}S_{x,r}^*(d_\Sigma) - d_{\bar{\Sigma}}\|_{L^1(B_G \times B_G)} = 0,
\end{equation}
where $\mathscr{F}(K_x)$ is the collection of cut measures $\bar{\Sigma}$ supported on $\F$ with $\bar{\Sigma}(\F)=\bar{\Sigma}(\Cut(G)) \leq K_x$.
\end{theorem}

To prove Theorem~\ref{thm:modCK}, we modify Sections $\S6$-$\S10$ of \cite{CKL1}. Many of our lemmas and propositions have direct analogues in \cite{CKL1}, and we organize them according to the section found in that paper. Before getting into these, let us highlight some of the important differences between \cite{CKL1} and the present article.
\begin{itemize}
    \item Instead of working with an arbitrary bounded open subset $U$, we focus on the unit ball $B_G$. This is only to avoid an additional unnecessary variable, and all our results could be stated with $U$ in place of $B_G$.
    \item The collection of half-spaces $HS$ is replaced with the compact collection $\F$ of constant normal cuts. This is an essential change because it is unknown, for general Carnot groups, if the generic tangents of a cut of locally finite perimeter are half-spaces (which is known to be true for the Heisenberg group by \cite{FSS03}).
    \item We require the  mass bound $\bar{\Sigma}(\F) \leq K_x$ in the infimum, while \cite{CKL1} requires no such bound. This is essential for our blowup argument occurring in Section~\ref{sec:blowup}, because it allows us to take a sequence of cut measures $\bar{\Sigma}_j$ belonging to a bounded set in $C^0(\F)^*$, which denotes the dual space of the continuous functions of $\F$, and apply weak* compactness.
\end{itemize}

The proof of Theorem~\ref{thm:modCK} occurs in the final subsection of this section (corresponding to $\S10$ in \cite{CKL1}) after a host of lemmas and propositions.
\vskip.3cm
\begin{center}\emph{For the remainder of this section, fix $G$, $\Sigma$, and $\F$ as in Theorem~\ref{thm:modCK}.}\end{center}

\subsection*{Controlling the total bad perimeter measure}
This subsection follows Sections $\S6$ and $\S7$ of \cite{CKL1}. 

Analogously to \cite{CKL1}, we establish a notion of ``bad'' points at which a cut is not well approximated by sets in $\F$ at a specified scale.  We define a total perimeter measure on the cuts, and we show that we have good control on the part of this measure attributed to the bad points.

We will first define a measurement of the distance of a cut $E$ to such cuts in a given collection $\F$ at a given scales.


For $x \in G$, we define $\F^{\{x\}} := \{F \in \F: x \in \partial^*F\}$ and define $\alpha: \Cut(G) \times G \times (0,\infty) \to (0,\infty)$ by
\begin{align*}
    \alpha(E,x,r) := d_{L^1(B_G)}(\F^{\{0\}},S_{x,r}^*(E)) = \inf_{F \in \F^{\{x\}}} \fint_{B_r(x)}|\uno_F-\uno_E|\, \dd \vol_G.
\end{align*}

Since $\F$ contains the $\Sigma$-generic tangents, we have the following lemma.



\begin{lemma}\label{lem:alpha-limit} 
The function $\alpha$ is continuous. Moreover,
for $\Sigma$-a.e.~every $E \in \Cut(G)$ and $\Per_E$-a.e.~$x \in G$, 
\[\lim_{r\to 0} \alpha(E,x,r)=0.\]
\end{lemma}

\begin{proof}

Consider the map $\Lambda:L^1_{\loc}(G)\times G \times (0,\infty) \to L^1_{\loc}(G)$ defined by $\Lambda(f,x,r)=S_{x,r}^*f$. We first show that $\Lambda$ is continuous. Indeed, fixing $f\in L^1_{\loc}(G)$, each slice $\Lambda(f, \cdot,\cdot): G\times (0,\infty) \to L^1_{\loc}(G)$ is continuous by the first paragraph of the proof of Lemma~\ref{lem:precompact}. 
To obtain the joint continuity, it suffices to prove that for any compact subset $K\subset G\times (0,\infty)$, the family of functions $\{\Lambda(\cdot, x,r): (x,r) \in K\}$ is uniformly equicontinuous.
To prove the latter property, we will exploit the fact that $L^1_{\loc}(G)$ is metrizable with the metric $d$ described in Section~\ref{subs cuts and perimeter}.

Suppose that $f_1,f_2 \in L^1_{\loc}(G)$, $R>0$ and $(x,r) \in K$. Then by a change of variables,
\begin{multline*}
\int_{B_R(0)} |\Lambda(f_1, x,r)-\Lambda(f_2, x,r)| \,\dd\vol_G 
= \int_{B_R(0)} |S_{x,r}^*(f_1)-S_{x,r}^*(f_1)| \,\dd\vol_G \\
= r^{-Q} \int_{B_{rR}(x)} |f_1-f_2| \,\dd\vol_G 
\le r^{-Q} p_{\lceil aR+b \rceil}(f_1-f_2) .
\end{multline*}
where $a=\sup\{s:(y,s)\in K\}$ and $b=\sup\{d_G(0,y):(y,s)\in K\}$.
Therefore
\begin{align*}
d(\Lambda(f_1, x,r) , \Lambda(f_2, x,r))
&\le \sum_{n=1}^\infty 1/n^2 \frac{ p_{\lceil an+b \rceil }(f_1-f_2) }{ r^Q+ p_{\lceil an+b \rceil }(f_1-f_2) }\\
&\le C_K d(f_1,f_2),
\end{align*}
where $C_K = (\lceil a\rceil+\lceil b\rceil)\max\{1,\sup\{1/s:(y,s)\in K\}\}$.
This yields that the functions $\{\Lambda(\cdot, g,r): (g,r) \in K\}$ are equi-Lipschitz from $(L^1_{\loc}(G),d)$ to itself, and thus uniformly equicontinuous. The continuity of $\alpha$ follows immediately from the continuity of $\Lambda$.

Since $\Sigma$ is an $\FP_{\loc}$ cut measure, we can take the $\Sigma$-generic set to have locally finite perimeter, that is we assume $E \in \FP_{\loc}(G)$. Now since $E$ has locally finite perimeter, we can apply Lemmas~\ref{lem:boundaries} and~\ref{lem:precompact} and find a full $\Per_E$-measure set of $x\in G$ satisfying the following: $x\in \partial^{str}_{mt}E$, the family $\{S_{x,r}^*(E)\}_{r \in (0,1]}$ is precompact in $L_{\loc}^1(G)$,
and $\F$ contains every tangent of $E$ at $x$. 

By compactness,
for each sequence $\{r_i\}$ with $r_i\to 0$, there is a subsequence $\{r_{i_j}\}$ such that the tangent of $E$ at $x$ along $\{r_{i_j}\}$, denoted by $F=F(E,x,\{r_{i_j}\})$, exists. Hence, $F \in \F$ by assumption on $x$. Furthermore, by Lemma~\ref{lem:strmtblowup}, since $x \in \partial^{str}_{mt}E$, also $x \in \partial^{str}_{mt}F$. By Lemma~\ref{lem:constantnormalperimeter}, $\partial^{str}_{mt}F = \partial^*F$, and hence $x \in \partial^*F$ and $F \in \F^{\{x\}}$. This implies that 
	\[\lim_{i_j\to\infty}\fint_{B_{r_{i_j}}(x)}|\uno_E-\uno_F|\,\dd\vol_G=0,\]
	from which
	\[\lim_{i_j\to\infty}\alpha(E,x,r_{i_j})=0\]
	follows. Since every sequence has a subsequence along which the limit is zero, we get the desired result: $\lim_{r\to 0}\alpha(E,x,r)=0$.
\end{proof}

The convergence of $\alpha(E,x,r)$ to zero may happen at very different rates. To quantify this uniformly, we introduce good and bad sets of cuts.

For scales $\eps,R > 0$ and $E \in \FP_{\loc}(G)$, we partition  $B_G$ into a set of ``bad points'' and a set of ``good'' points.  We declare $x\in B_G$ ``bad'' if $x$ is close to $G\setminus B_G$ or if $\alpha(E,x,r)$ is large, relative to scales $R$ and $\eps$, respectively.
\begin{definition} For   $\eps,R > 0$ and $E \in \FP_{\loc}(G)$, we define the following sets:
\begin{itemize}
    \item $\Bad_{\eps,R}(E) := \{x \in B_G: d_G(x,G \setminus B_G) < R$ or $\alpha(E,x,r) > \eps$ for some $r \in (0,R]\}$, 
    \item $\Good_{\eps,R}(E) := B_G \setminus$ $\Bad_{\eps,R}(E)$
\end{itemize}
We use this partition to declare pairs  $(E, x)\in\Cut(G)\times B_G$  ``bad'' or ``good''.
\begin{itemize}
    \item $\Bad_{\eps,R} := \{(E,x) \in \FP_{\loc}(G) \times B_G: x \in$ $\Bad_{\eps,R}(E)\}$, and
    \item $\Good_{\eps,R} := (\FP_{\loc}(G) \times B_G) \setminus \Bad_{\eps,R}$.
\end{itemize}
\end{definition}

The continuity of $\alpha$ proven in Lemma~\ref{lem:alpha-limit} implies that $\Bad_{\eps,R} = \{(E,x): d(x,G\setminus B_G)< R\}\cup \bigcup_{r\le R}\{(E,x) : \alpha(E,x,r)>\eps\} $ is open.
Below we show that the perimeter measure of $\Bad_{\eps,R}$ goes to zero as $R$ goes to zero.





The next lemma is exactly the same as \cite{CKL1}*{Lemma~6.6}. 
We need the map $\Per \mres \Bad_{\eps,R}$ to be weakly $L^1$ for the integral in~\eqref{def total bad perimeter measure} to be well defined.
See~\cite{CKL1}*{Section~5.1} for the definition of \emph{weakly $L^1$} and further details.

\begin{lemma}\label{lem: weakly L1} 
For every $\eps,R>0$, 
the map $\Per \mres \Bad_{\eps,R}: \FP_{\loc}(G) \to$ $\Radon(B_G)$ defined by $E \mapsto \Per_E \mres \Bad_{\eps,R}(E)$ is weakly $L^1$.
\end{lemma}
\begin{proof}\footnote{We present the proof here for the convenience of the reader. The proof is omitted in the journal version of this paper.}
	For $k\in\mathbb{N}$, let $\Phi_k:\FP_{\loc}(G)\times B_G\to\R$ denote a continuous function satisfying
	\begin{enumerate}
		\item[(i)] $0\le\Phi_k\le 1$;
		\item[(ii)] $\Phi_k\equiv 1$ on $\{(E,x)|d((E,x),\Good_{\eps,R})\ge 1/k\}$; and
		\item[(iii)] $\Phi_k\equiv 0$ on $\{(E,x)|d((E,x), \Good_{\eps,R})\le 1/(k+1)\}$,
	\end{enumerate}
	where in (ii) and (iii), the distance on the product space is taken to be the sum of the individual factor distances.
	
	Fix $\phi\in C_c(B_G)$ and define $\Psi_k:\FP_{\loc}(G)\to\R$ by 
	\[\Psi_k(E)=\int_{B_G}\phi\cdot\Phi_k(E,\cdot)\,\dd \Per_E.\]
	The map $\Psi_k$ is Borel measurable since it is the pointwise limit of a sequence of measurable functions $\{\Psi_{k,l}\}$ obtained by approximating the  map $E\mapsto \phi(\cdot)\Phi(E,\cdot)$ by simple functions and each $\Psi_{k,l}$ is measurable.  
	
	For fixed $E\in\FP_{\loc}(G)$, the compact subsets 
	\[\supp(\Phi_k(E,\cdot))\subset \Bad_{\eps,R}(E)\]
	exhaust the open set $\Bad_{\eps,R}(E)$.  To see that this claim holds, note that for each compact subset $K\subset\Bad_{\eps,R}(E)$, $K$ has positive distance from the closed set $\Good_{\eps,R}$.  Thus, for $k$ taken small enough, $K\subset\supp(\Phi_k(E, \cdot))$.
	If follows that the mass of the difference measure
	\[\Per_E\mres\Bad_{\eps,R}(E)-\Phi_k(E,x)\Per_E\]
	goes to zero as $k\to\infty$. 
	Thus, for each $E\subset\E$, the integrals
	\[\Psi_k(E):=\int_{B_G}\phi\cdot\Phi_k(E,x)\,\dd \Per_E\]
	converge to 
	\[\int_{B_G}\phi~ d(\Per_E\mres\Bad_{\eps,R}(E))\]
	as $k\to\infty$.
	The map $E\mapsto \int_{B_G} \phi~d(\Per_E\mres\Bad_{\eps,R}(E))$ is a pointwise limit of Borel measurable functions, and is therefore Borel measurable. Since $\phi$ is arbitrary  it follows that the map
	\[\Per\mres\Bad_{\eps,R}:\FP_{\loc}(G)\to\Radon(B_G)\]
	is weakly measurable.
	Finally, under the assumption that $\Sigma$ is an $\FP_{\loc}$ cut measure,
	\[\int_{\Cut(G)}\Per_E(B_G)\,\dd \Sigma<\infty,\]
	implies that $\Per\mres\Bad_{\eps, R}$ is weakly $L^1$.
\end{proof}

\begin{definition}
 Define the \emph{total perimeter measure} $\lambda \in \Radon(B_G)$ by
\begin{align*}
    \lambda := \int_{\FP_{\loc}(G)} \Per_E \mres B_G \: \,\dd \Sigma(E),
\end{align*}
the \emph{total bad perimeter measure} $\lambda^{\Bad}_{\eps,R} \in \Radon(B_G)$ by
\begin{equation}\label{def total bad perimeter measure}
    \lambda^{\Bad}_{\eps,R} := \int_{\FP_{\loc}(G)} \Per_E \mres \Bad_{\eps,R}(E) \: \,\dd\Sigma(E),
\end{equation}
and the \emph{total good perimeter measure} $\lambda^{\Good}_{\eps,R} \in \Radon(B_G)$ by $\lambda^{\Good}_{\eps,R} := \lambda - \lambda^{\Bad}_{\eps,R}$.\end{definition}

\begin{lemma}\label{lem:SmallBadPerimeter}
For all $\eps > 0$, $\lim_{R \to 0} \lambda^{\Bad}_{\eps,R}(B_G) = 0$.
\end{lemma}

\begin{proof}
For $\Sigma$-a.e.~$E \in \FP_{\loc}(G)$ we have $\lim_{r\to 0}\Per_E(\Bad_{\eps,r}(E))=0$.
Indeed, by Lemma~\ref{lem:alpha-limit}, for $\Sigma$-a.e.~$E \in \FP_{\loc}(G)$ we have that $\lim_{r\to 0}\alpha(E,x,r)=0$ for $\Per_E$-a.e.~$x \in G$.
Hence, for any such $E$, the family of open subsets $\{\Bad_{\eps,r}(E)\}_r$ of $B_G$ is decreasing to a $\Per_E$-null set, as $r\to 0$.
We obtain $\lim_{r\to 0}\Per_E(\Bad_{\eps,r}(E))=0$ by continuity from above for measures.


Consider the function $\Lambda_R:\FP_{\loc}(G)\to \R$ defined by 
\begin{equation*}
\Lambda_R(E):= \Per_E \mres \Bad_{\eps,R}(E) (B_G) = \Per_E(\Bad_{\eps,R}(E)) .
\end{equation*} 
By the previous paragraph, for $\Sigma$-a.e.~$E$ we have $\lim_{R\to 0} \Lambda_R(E)=0$. Also, $\Lambda_R(E) \leq \Per_E(B_G)$, where the map $E \mapsto \Per_E(B_G)$ is in $L^1(\Sigma)$ since $\Sigma$ is a $\FP_{\loc}$ cut measure. Then, by dominated convergence, $\lim_{R\to 0}\Lambda_R = 0$ in $L^1(\Sigma)$. Thus,
\[
\lim_{R\to 0}\lambda^{\Bad}_{\eps,R}(B_G) = \lim_{R\to 0} \int_{\Cut(G)} \Lambda_R(E) \,\dd\Sigma(E) = 0.
\]
\end{proof}

In \cite{CKL1}, the authors define a set $U_{\delta, \eps}$ of points, together with parameters $r_0, r_1, R_0$, where $\lambda^{\Bad}$ has small density up to scale $r_1$ and $\lambda$ bounded density up to scale $R_0$. We gather these in the following definition.

\begin{definition}\label{def:regular-set}
A $\vol_G$-measurable subset $U_{\delta, \eps} \subset B_G$ is called {\em$(\delta,\eps)$-regular at scales} $(r_0,r_1,R_0)$ if the following three properties hold:

\begin{equation} \label{eq:Uvolume}
	\vol_G(B_G\setminus U_{\delta,\eps})<2\delta(1+\lambda(B_G)),
	\end{equation}
	
\begin{equation}\label{eq:BoundedPerimeterRatio}
	\frac{\lambda(B_r(x))}{\vol_G(B_r(x))}<\delta^{-1}\qquad \text{ if }x\in U_{\delta,\eps},~r\le r_0,
\end{equation}

\begin{equation}\label{eq:SmallBadPerimeter}
	\frac{\lambda^{\Bad}_{\eps,R_0}(B_r(x))}{\vol_G(B_r(x))}<\eps,\qquad \text{ if }x\in U_{\delta,\eps},~ r\le r_1.
\end{equation}
\end{definition}

\begin{remark} When $\Sigma$ is a cut measure corresponding to an $L$-Lipschitz function $f:G \to L^1$, inequality \eqref{eq:BoundedPerimeterRatio} is true with $\delta \lesssim \frac{1}{L}$ for all $r>0$; see \cite{CKL1}*{Proposition~5.10}. However, we wish to state the results for general $\FP_{\loc}$ cut measures.

\end{remark}


Due to a need to take a limit as $r\to 0$, we require a collection of such regular sets that are nested. While this idea is implicitly present in the proof of \cite{CKL1}*{Theorem~10.20}, we wish to make it explicit here.

\begin{lemma}\label{lem:nestedseq} Fix  $\delta>0$. Then, for every sequence of $\eps_j \searrow 0$, there exist sequences of scales $r_0(j),r_1(j),R_0(j)>0$ and sets $U_{\delta, \eps_j}$ such that the following hold:

\begin{enumerate}
    \item Each $U_{\delta,\eps_j}$ is $(\delta,\eps_j)$ regular at scales $(r_0(j),r_1(j),R_0(j))$.
    \item The sets are nested, that is 
        $U_{\delta,\eps_1}  \supset  U_{\delta,\eps_2} \supset U_{\delta,\eps_3} \supset \dots$.
\end{enumerate}
In particular, for $U_\delta := \bigcap_j U_{\delta,\eps_j}$ we have  $\vol_G(B_G \setminus U_\delta) \leq 2\delta(1+\lambda(B_G))$.
\end{lemma}
As in \cite{CKL1}*{Proposition~7.5}, the existence of $U_{\delta,\eps},R_0,r_0,r_1$ for fixed parameters $\delta,\eps$ follows from Lemma~\ref{lem:SmallBadPerimeter} and a straightforward application of measure differentiation with respect to doubling measures \cite{Heinonen}*{Section 2.7}. To obtain the nested property, we need to choose the sets slightly more carefully.




\begin{proof}
    In this proof, all sets $U$ (with sub- and/or super-scripts) are claimed to be Borel measurable.
	By Lebesgue decomposition theorem, there exists 
	a set
	$U\subset B_G$ such that $\vol_G(B_G\setminus U)=0$ and $\lambda$ is absolutely continuous with respect to $\vol_G$ on $U$.
	Furthermore, since \[\int_U\frac{ \dd\lambda}{\,\dd\vol_G}\,\dd\vol_G\le \lambda(B_G)<\infty,\]
	there exists $U_1\subset U$ such that \[\vol_G(B_G\setminus  U_1)=\vol_G(U\setminus U_1)<2\delta\lambda(B_G)\]
	and \[\frac{\dd\lambda}{\,\dd\vol_G}<\frac{\delta^{-1}}{2}\quad\text{ on }U_1.\]
	By measure differentiation \cite{Heinonen}*{Section 2.7}, for $\vol_G$-a.e.~$x\in U_1,$
	\[\lim_{r\to 0}\frac{\lambda(B_r(x))}{\vol_G(B_r(x))}=\frac{\dd\lambda}{\,\dd\vol_G}(x).\]
	Thus there exits $r_0'>0$ and $U_2\subset U_1$ such that $\vol_G(U_1\setminus U_2)<\delta$, and
	\[\frac{\lambda(B_r(x))}{\vol_G(B_r(x))}<\delta^{-1}\]
	whenever $x\in U_2$ and $0<r\le r_0'$. 
	We will then set $r_0(j):=r_0'$ for all $j\in\N$.
	
	
	Now fix $j \in \N$. By Lemma~\ref{lem:SmallBadPerimeter}, there exists $R_0 = R_0(j)$ such that
	\[
	\lambda^{\Bad}_{\eps_j,R_0}(B_G) < \delta \eps_j 2^{-j-2},
	\]
	and thus there exists $U_3^j\subset U$ such that \[\vol_G(B_G\setminus U_3^j)=\vol_G(U\setminus U_3^j)<\delta2^{-(j+1)}\] and
	\[\frac{\dd\lambda^{\Bad}_{\eps_j, R_0}}{\,\dd\vol_G}<\eps_j/2\quad\text{ on }U_3^j.\]
	Finally, using measure differentiation as above, there exists $U_4^j\subset U_3^j$ and $r_1(j)>0$ 
	such that $\vol_G(U_3^j\setminus U_4^j)<\delta2^{-(j+1)}$ and 
	\[\frac{\lambda^{\Bad}_{\eps_j, R_0}(B_r(x))}{\vol_G(B_r(x))}<\eps_j\]
	whenever $x\in U_4^j$ and $0<r\le r_1(j)$.
	By choosing $U_{\delta, \eps_j}:=U_2\cap \bigcap_{k=1}^j U_4^k$, the desired result is obtained.
\end{proof}

\subsection*{Collections of good and bad cuts} This subsection corresponds to Section $\S8$ of \cite{CKL1}.   
\vskip.1cm
\emph{From this point till the proof of Theorem \ref{thm:modCK} at the end of this section, the symbols $\delta,\eps,r_0,r_1,R_0 > 0$ denote fixed positive constants and $U_{\delta,\eps}$ denotes a fixed $\vol_G$-measurable subset of $B_G$ that is $(\delta,\eps)$-regular at scales $(r_0,r_1,R_0)$.}
\vskip.1cm

For given scales and locations, we partition the collection of locally finite-perimeter cuts into good cuts and bad cuts, and we establish size estimates on these subcollections.


\begin{definition}
Following \cite{CKL1}*{page 1378}, for given $x\in B_G$ and $r>0$, we decompose $\FP_{\loc}(G)$ into a collection of \emph{good cuts} $\G(x,r,\eps,R_0)$ and a collection of \emph{bad cuts} $\B(x,r,\eps,R_0)$, where 
\begin{equation}\G(x,r,\eps,R_0):=\left\{E\in\FP_{\loc}(G)~|~\overline{B_r(x)}\cap\Good_{\eps,R_0}(E)\ne\emptyset\right\}
\end{equation}
and 
\begin{equation}
\B(x,r,\eps,R_0):=\FP_{\loc}(G)\setminus\G(x,r,\eps,R_0).
\end{equation}
\end{definition}
\begin{prop}\label{prop:boundbad}
	If $x\in U_{\delta, \eps}$ and $r\in (0, r_1)$, then
	\[\frac{1}{\vol_G(B_r(x))}\int_{\B(x,r,\eps,R_0)}\Per_E(B_r(x))\,\dd\Sigma(E)<\eps.\]
\end{prop}

\begin{remark}\label{rmk:ErrorProp8.2}
    There is a small error in the corresponding statement \cite{CKL1}*{Proposition~8.2}, where a $\delta$ substitutes $\eps$. We provide a detailed proof to fix this. While for them the error was mostly inconsequential, for us it is crucial to rectify it. Indeed, it subtly affects the limiting process of sending $\eps\to 0$.  The issue will appear later in the proof of Theorem~\ref{thm:modCK}, which was modified from \cite{CKL1}*{Theorem~10.2}. We return to this point in Remark~\ref{rmk:correctionCK10.2}.
\end{remark}

\begin{proof}[Proof of Proposition~\ref{prop:boundbad}.]
	If $E\in\B:=\B(x,r,\eps,R_0)$,  
	then by the definition of $\B $ we have $\overline{B_r(x)}\subset\Bad_{\eps,R_0}(E)$.
	Hence, we have  
	\begin{align*}
	\int_{\B} \Per_E(B_r(x)) \,\,\dd\Sigma(E)
	&= \int_{\B} \Per_E(B_r(x)\cap \Bad_{\eps,R_0}(E)) \,\,\dd\Sigma(E) \\
	&\le  \int_{\Cut(G))} 
	\Per_E(B_r(x)\cap \Bad_{\eps,R_0}(E)) \,\dd\Sigma(E) \\
	&\stackrel{\eqref{def total bad perimeter measure}}{=} \lambda^{\Bad}_{\eps,R_0}(B_r(x)) \\
	&\le \eps \vol_G(B_r(x)) ,
	\end{align*}
	where the latter inequality uses~\eqref{eq:SmallBadPerimeter},
	together with the assumptions $x\in U_{\delta,\eps}$ and $0<r<r_1$.
\end{proof}


\begin{lemma} \label{lem:boundgood}
There exists $\eps_0 > 0$ such that, for every $E \in \FP_{\loc}(G)$, $x \in G$, and $r > 0$, if $\alpha(E,x,r) < \eps_0$, then $\Per_{S_{x,r}^*(E)}(B_G) \geq \tfrac{c}{2}$, where $c$ is the constant from Lemma~\ref{lem:constantnormalperimeter}.
\end{lemma}
\begin{proof}
Suppose the lemma is false. Then we can find sequences $(E_n)_n \subset \FP_{\loc}(G)$, $(x_n)_n \subset G$, and $(r_n)_n \subset(0,+\infty)$ such that, setting $E'_n := S_{x_n,r_n}^*(E_n)$, one has
\begin{equation*}
d_{L^1(B_G)}(\F^{\{0\}},E'_n) \stackrel{\rm def}{=} \alpha(E_n,x_n,r_n) \to 0, \quad \text{ as } n \to \infty,
\end{equation*}
and
\begin{equation*}
\Per_{E'_n}(B_G) < \tfrac{c}{2} , \quad \forall n\in N.
\end{equation*}
By the perimeter bound and Lemma~\ref{lem:BVcompact}, we may pass to a subsequence and assume $E_n' \to E$ in $L^1(B_G)$, for some $E\subseteq B_G$.
Since $\alpha$ is continuous by Lemma~\ref{lem:alpha-limit}, we have  $d_{L^1(B_G)}(\F^{\{0\}},E) = 0$.
Since $\F$ is assumed to be compact, $\{F\cap B_G:F\in \F^{\{0\}}\}$ is compact in $L^1(B_G)$ by Theorem~\ref{thm:Fkcompact} 
and thus there is $F\in\F^{\{0\}}$ such that $E=F\cap B_G$.
By the lower semicontinuity of $E \mapsto \Per_E$ with respect to $L^1$-convergence \cite{FSS03}*{Proposition~2.12}, this implies that there exists $F \in \F^{\{0\}}$ with $\Per_F(B_G) \leq \tfrac{c}{2}$,
in contradiction to Lemma~\ref{lem:constantnormalperimeter}.
\end{proof}

\begin{prop}\label{prop:boundgood}
Consider the constants $\eps_0$ and $c$   from Lemma~\ref{lem:boundgood}.
If $\eps<\frac{\eps_0}{2}$,
  $r<\min\left(\frac{r_0}{2},R_0\right)$, and $x \in U_{\delta,\eps}$, then
	\[\Sigma(\G(x,r,\eps,R_0))\le C_0r\delta^{-1},\]
with $C_0 := \tfrac{2^{Q+1}}{c}$.
\end{prop}

\begin{proof}
By definition of $\G := \G(x,r,\eps,R_0)$, if $E\in\G$, then there exists $x'\in\overline{B_r(x)}\cap\Good_{\eps,R_0}(E)$. 
By definition of $\Good_{\eps,R_0}(E)$ together with the assumption that $r<R_0$ and $\eps<\eps_0$,
\[\alpha(E,x',r)\le\eps < \eps_0.\]
Thus, by Lemma~\ref{lem:boundgood},
\begin{equation*}
r^{1-Q}\Per_E(B_{2r}(x)) \geq r^{1-Q}\Per_E(B_{r}(x')) = \Per_{S_{x',r}^*(E)}(B_G) \geq \tfrac{c}{2}.
\end{equation*}
Since this holds for every $E \in \G$, we have
\begin{align*}
    \lambda(B_{2r}(x)) \geq \int_\G \Per_E(B_{2r}(x))\, \dd \Sigma(E) \geq \frac{cr^{Q-1}}{2}\Sigma(\G).
\end{align*}
Then we apply \eqref{eq:BoundedPerimeterRatio} and get
\begin{align*}
    \Sigma(\G) \leq \frac{2r^{1-Q}}{c}\lambda(B_{2r}(x)) = \frac{2^{Q+1}}{c} \frac{r\lambda(B_{2r}(x))}{\vol_G(B_{2r}(x))} \leq C_0r\delta^{-1}.
\end{align*}
\end{proof}

\subsection*{The approximating cut measure supported on \texorpdfstring{$\F$}{F}}
This subsection corresponds to Section $\S9$ of \cite{CKL1}. 
Following \cite{CKL1}, we construct a Borel map that assigns to each good cut $E\in\G$ a sufficiently close set in $\F$. In the subsequent section we will use this map to define a cut measure that is supported on $\F$.

\begin{lemma}\label{lem:goodsetstoF} 
Let $x \in B_G$, $r > 0$, and set $\G := \G(x,r,\eps,R_0)$. If $r < \frac{R_0}{2}$, then there exists a Borel map $\gamma: \G \to \F$ such that, for every $E \in \G$, there exists $x' \in \overline{B_{r}(x)}$ with
\begin{equation}\label{lem:goodsetstoF_eq} 
    \fint_{B_{2r}(x')} |\uno_E - \uno_{\gamma(E)}| \: \,\dd\vol_G < 4\eps.
\end{equation}
\end{lemma}

\begin{proof} 
Since the map $\F \to L^1(B_r(x))$, $E \mapsto \uno_{E \cap B_{4r}(x)}$,
 is continuous and $\F$ is compact, 
 there exists a finite collection of cuts $\{F_1,\dots, F_N\} \subset \F$ so that for each cut $F \in \F$ 
\[\min_{i=1,\dots, N}\int_{B_{4r}(x)} |\uno_F-\uno_{F_i}| \,\dd\vol_G < (2r)^Q \eps ,\]
where $Q$ is the homogeneous dimension of $G$. 
Further, by compactness there exist $x_1,\dots, x_M \in \overline{B_{r}(x)}$, so that for any $y\in \overline{B_{r}(x)}$ and some $j=1,\dots, M$ we have
\[\frac{\vol_G(B_{2r}(x_j)\Delta B_{2r}(y))}{\vol_G(B_{2r}(y))} \leq \eps .\]
Here, $A\Delta B$ denotes the symmetric difference.

By definition of $\G$, for each $E \in \G$, there exists a $y \in \overline{B_r(x)}\cap \Good_{\eps,R_0}(E)$. By definition of $\Good_{\eps,R_0}(E)$ and since $2r<R_0$, there exists an $F_E\in \F^{\{y\}}$ with 
\[\fint_{B_{2r}(y)} |\uno_E-\uno_{F_E}| \,\dd\vol_G < 2\eps.\]

Collecting these facts, for every $E\in \G$ there are $j\in\{1,\dots, M\}$ and $i\in\{1,\dots, N\}$  so that,
for some $y\in \overline{B_{r}(x)}$ and $F_E\in \F^{\{y\}}$, we have
\begin{align*}
\fint_{B_{2r}(x_j)} |\uno_E-\uno_{F_i}| \,\dd\vol_G 
&\le \fint_{B_{2r}(y)} |\uno_E-\uno_{F_E}| \,\dd\vol_G
    + \frac{\vol_G(B_{2r}(x_j)\Delta B_{2r}(y))}{\vol_G(B_{2r}(y))} \\
&\hspace{2.8cm}    + \frac1{(2r)^Q} \int_{B_{4r}(x)} |\uno_{F_i}-\uno_{F_E}| \,\dd\vol_G \\
&<4\eps.
\end{align*}

Define $U_{1,1} = \{E \in \G : \fint_{B_{2r}(x_1)} |\uno_E-\uno_{F_1}| \,\dd\vol_G < 2\eps\}$, 
and recursively for all $(k,l) \in \{1,\dots, N\}\times \{1,\dots, M\}$ following the lexicographic total order of $\N\times\N$, define 
\[
U_{k,l} = \left\{E \in \G : \fint_{B_{2r}(x_k)} |\uno_E-\uno_{F_l}| \,\dd\vol_G < 2\eps\right\} \setminus \bigcup_{(i,j) < (k,l)} U_{i,j} .
\]
By construction and the previous paragraph, $\G$ is contained is the disjoint union of these Borel sets: $\G \subset \bigsqcup_{i=1}^N \bigcup_{j=1}^M U_{i,j}$. We set $\gamma(E) := F_l$ when $E \in U_{k,l}$ and note that the conclusion holds with $x':=x_k$.
\end{proof}


\subsection*{The Proof of Theorem~\ref{thm:modCK}} This subsection corresponds to Section $\S10$ of \cite{CKL1}.
We prove Theorem~\ref{thm:modCK} by first establishing estimates on the good and bad parts of the cut metric $d_\Sigma$.

\vskip.3cm
In the lemma below, we estimate the good part of the cut metric. 

\begin{lemma} \label{lem:goodboundmain}
Let $x \in U_{\delta,\eps}$ and $r>0$, and set $\G := \G(x,r,\eps,R_0)$. Let $\gamma: \G \to \F$ be the map from Lemma~\ref{lem:goodsetstoF}. Let $\eps_0$ be the constant from Lemma~\ref{lem:boundgood} and $C_0$ the constant from Proposition~\ref{prop:boundgood}. If $\eps < \eps_0$, then for the pushforward measure
$\hat{\Sigma}:=\gamma_\#(\Sigma \mres \G)$, we have
\begin{align*}
\|d_{\Sigma \mres \G} - d_{\hat{\Sigma}}\|_{L^1(B_r(x) \times B_r(x))} \leq 16C_0r\eps\delta^{-1}\vol_G(B_r(x))^2.
\end{align*}
\end{lemma}

\begin{proof}
The proof uses nearly the same estimates as \cite{CKL1}.
\begin{align*}
\int_{B_r(x) \times B_r(x)}& |d_{\Sigma \mres \G}(a,b) - d_{\hat{\Sigma}}(a,b)|\,\dd\vol_G(a) \,\dd\vol_G(b) \\
&\leq \int_{B_r(x) \times B_r(x)} \int_{\FP_{\loc}(G)} |\uno_E(a)-\uno_E(b) - \uno_{\gamma(E)}(a) \\
&\hspace{4cm} + \uno_{\gamma(E)}(b)|\,\dd\, \Sigma \mres \G(E) \,\dd\vol_G(a) \,\dd\vol_G(b) \\
&\leq  \int_{\FP_{\loc}(G)} \int_{B_r(x) \times B_r(x)} |\uno_E(a)-\uno_{\gamma(E)}(a)| + |\uno_E(b) - \uno_{\gamma(E)}(b)|\\&\hspace{5.5cm} \,\dd\vol_G(a) \,\dd\vol_G(b) \,\dd\, \Sigma \mres \G(E) \\
&\stackrel{\eqref{lem:goodsetstoF_eq}}{\leq} \int_{\FP_{\loc}(G)} 16\eps \vol_G(B_r(x))^2 \,\dd\, \Sigma \mres \G(E) \\
&= 16\eps \vol_G(B_r(x))^2 \Sigma(\G).
\end{align*}
The claim then follows from Proposition~\ref{prop:boundgood}.
\end{proof}

Next, we estimate the bad part of the cut metric. For this claim, we introduce the \emph{Poincar\'e constant of $G$}. The Poincaré constant $\tau$ 
is a constant that satisfies for all functions of bounded variation $f\in {\rm BV}_{\loc}(G)$ and all balls $B_r=B_r(x) \subset G$ the following inequality:
	\[
		\int_{B_r\times B_r} |f(x_1)-f(x_2)| \,\dd (\vol_G\times\vol_G)(x_1,x_2)
			\le r\tau |Df|(B_r(x))\vol_G(B_r(x)).
	\]
We will need only the case when $f=\uno_E$ for $E \in \FP_{\loc}$, for which $|Df| = \Per_E$. For the definition of functions of bounded variation, see \cite{Ambrosio}. For the inequality in this form and for references to various places in which the inequality has been proven, see \cite{franchiPI}.

%



\begin{lemma}[Estimating the bad part of the cut metric] \label{lem:badboundmain}
	Let $\tau>0$ be the Poincaré constant of $G$.
	If $x\in U_{\delta,\eps}$
	and $0<r< r_1$, 
	then
	\begin{equation*}
		\left\| d_{\Sigma\mres\B(x,r,\eps,R_0)} \right\|_{L^1(B_r(x)\times B_r(x))} \le \tau\,\eps\, r\, \vol_G(B_r(x))^2 .
	\end{equation*}
\end{lemma}
\begin{proof}
	Set $\B := \B(x,r,\eps,R_0)$. A direct computation shows that
	\begin{align*}
	&\hspace{-1cm}\left\| d_{\Sigma\mres\B} 
		\right\|_{L^1(B_r(x)\times B_r(x))} \\
	&\overset{(a)}= \int_{B_r(x)\times B_r(x)} \int_{\B} |\uno_E(x_1)-\uno_E(x_2)| \,\dd \Sigma(E) \,\dd (\vol_G\times\vol_G)(x_1,x_2) \\
	&\overset{(b)}=  \int_{\B} \int_{B_r(x)\times B_r(x)} |\uno_E(x_1)-\uno_E(x_2)| \,\dd (\vol_G\times\vol_G)(x_1,x_2) \,\dd \Sigma(E)  \\
	&\overset{(c)}\le \int_{\B} 
		\tau r\, \Per_E(B_r(x))\, \vol_G(B_r(x)) 
		\,\dd \Sigma(E)  \\
	&\overset{(d)}\le \tau\,\eps\, r\, \vol_G(B_r(x))^2 ,
	\end{align*}
	where $(a)$ is by definition,
	$(b)$ is an application of Fubini Theorem (which we can apply because $\Sigma$ is $\sigma$-finite 
	and because the integrand function is non-negative), 
	$(c)$ is an application of the Poincaré inequality above, and
	$(d)$ is an application of Proposition~\ref{prop:boundbad}. 
\end{proof}



Now, we are in the position to prove the main technical tool, Theorem~\ref{thm:modCK}. 

\begin{remark}\label{rmk:correctionCK10.2}
Before we embark on the proof of Theorem~\ref{thm:modCK}, we describe one small, yet crucial difference to \cite{CKL1}*{Theorem~10.2}. There, the proof uses \cite{CKL1}*{Propositon 8.2}, which as indicated above in Remark~\ref{rmk:ErrorProp8.2} contained a small error. This yielded an additive $\tau\delta$-term, instead of the $\tau\eps$-term we have in Equation \eqref{eq:correctedbound}. 
Consequently, the proof in \cite{CKL1}*{Theorem~10.2} had a limiting process which involved sending $\eps \to 0$ followed by sending $\delta \to 0$. Indeed, with the correction, it suffices to send $\eps \to 0$. This is of crucial importance for us, while in the original proof, the double limit was also allowed.
\end{remark}

\begin{proof}[Proof of Theorem~\ref{thm:modCK}]
Let $\delta > 0$ and let $(\eps_j)_{j=1}^\infty$ be any sequence decreasing to 0 with $\eps_1 < \eps_0$, where $\eps_0$ is the constant from Lemma~\ref{lem:boundgood}. Fix sets $U_{\delta,\eps_j}$ that are $(\delta,\eps_j)$-regular at scales $(r_0(j),r_1(j), R_0(j))$, which are afforded to us by Lemma~\ref{lem:nestedseq}. Let $x \in U_\delta := \bigcap_{j=1}^\infty U_{\delta,\eps_j}$.

Fix $j \in \N$ and $r <\min\{ r_0(j),r_1(j)\}$, and define $\G_j := \G(x,r,\eps_j,R_0(j))$ and $\B_j := \B(x,r,\eps_j,R_0(j))$. Let $\gamma: \G_j \to \F$ be the map defined in Lemma~\ref{lem:goodsetstoF} and let $\hat{\Sigma}_j$ be the pushforward of $\Sigma \mres \G_j$ under $\gamma$ as in Lemma~\ref{lem:goodboundmain}. By Proposition~\ref{prop:boundgood} and the fact that $\F$ is translation and dilation invariant, the rescaled cut measure $\tfrac{1}{r}S_{x,r}^*(\hat{\Sigma}_j)$ belongs to $\mathscr{F}(C_0\delta^{-1})$ where $C_0$ depends only on the group $G$ (recall that $\mathscr{F}(K)$ is the collection of cut measures $\overline{\Sigma}$ supported on $\F$ with $\overline{\Sigma}(\Cut(G)) \leq K$). Here, the rescaled cut measure $\tfrac{1}{r}S_{x,r}^*(\hat{\Sigma}_j)$ is defined by $\tfrac{1}{r}S_{x,r}^*(\hat{\Sigma}_j)(E) := \tfrac{1}{r}\hat{\Sigma}_j(\delta_{1/r}(x^{-1}E))$, and it is straightforward to check that $d_{\tfrac{1}{r}S_{x,r}^*(\hat{\Sigma}_j)} = \tfrac{1}{r}S_{x,r}^*(d_{\hat{\Sigma}_j})$. Then we have, by Lemmas~\ref{lem:goodboundmain} and~\ref{lem:badboundmain},
\begin{align}
    \inf_{\bar{\Sigma} \in \mathscr{F}(C_0\delta^{-1})} \|\tfrac{1}{r}S_{x,r}^*(d_\Sigma) - & d_{\bar{\Sigma}}\|_{L^1(B_G \times B_G)} 
    \leq \|\tfrac{1}{r}S_{x,r}^*(d_\Sigma) - \tfrac{1}{r}S_{x,r}^*(d_{\hat{\Sigma}_j})\|_{L^1(B_G \times B_G)} \notag \\
    &= \frac{1}{r\vol_G(B_r(x))^2}\|d_\Sigma - d_{\hat{\Sigma}_j}\|_{L^1(B_r(x) \times B_r(x))} \notag\\
    &\leq \frac{1}{r\vol_G(B_r(x))^2}\|d_{\Sigma \mres \G} - d_{\hat{\Sigma}_j}\|_{L^1(B_r(x) \times B_r(x))} \notag\\
    & \hspace{15pt} + \frac{1}{r\vol_G(B_r(x))^2}\|d_{\Sigma \mres \B}\|_{L^1(B_r(x) \times B_r(x))} \notag \\
    &\leq 16C_0\eps_j\delta^{-1} + \tau\eps_j. \label{eq:correctedbound}
\end{align}
Letting $r\to0$, this gives us
\begin{align*}
    \limsup_{r \to 0} \inf_{\bar{\Sigma} \in \mathscr{F}(C_0\delta^{-1})} \|\tfrac{1}{r}S_{x,r}^*(d_\Sigma) - d_{\bar{\Sigma}}\|_{L^1(B_G \times B_G)} \leq 16C_0\eps_j\delta^{-1} + \tau\eps_j.
\end{align*}
Since $\eps_j \to 0$ and the left-hand-side does not depend on $j$, this implies
\begin{align*}
    \limsup_{r \to 0} \inf_{\bar{\Sigma} \in \mathscr{F}(C_0\delta^{-1})} \|\tfrac{1}{r}S_{x,r}^*(d_\Sigma) - d_{\bar{\Sigma}}\|_{L^1(B_G \times B_G)} = 0.
\end{align*}
Thus, \eqref{eq:modCK} holds for every $\delta>0$ and every $x \in U_\delta$, with $K_x := C_0\delta^{-1}$. By Lemma~\ref{lem:nestedseq}, the set $\bigcup_{\delta>0} U_\delta$ has full $\vol_G$-measure in $B_G$, and thus \eqref{eq:modCK} holds for $\vol_G$-a.e.~$x \in B_G$.

To extend to almost every $x \in G$, we use a simple translation trick. Let $g \in G$ be arbitrary, and define a new cut measure $\Sigma_g$ on $G$ by $\Sigma_g(E) := \Sigma(g^{-1}E)$. Then since $\F$ is translation invariant, the pair $(\Sigma_g,\F)$ satisfies all the hypotheses of the theorem. Thus, by the preceding argument, for almost every $x \in B_G$, there exists $K_x < \infty$ such that
\begin{align*}
     \lim_{r \to 0} \inf_{\bar{\Sigma} \in \mathscr{F}(K_x)} \|\tfrac{1}{r}S_{x,r}^*(d_{\Sigma_g}) - d_{\bar{\Sigma}}\|_{L^1(B_G \times B_G)} = 0.
\end{align*}
A direct computation shows that $S_{x,r}^*(d_{\Sigma_g}) = S_{gx,r}^*(d_\Sigma)$, and so \eqref{eq:modCK} holds for almost every $x \in gB_G$. Since $G$ can be covered by a countable collection of the form $\{g_iB_G\}_{i\in \N}$, for some $\{g_i\}_{i \in \N} \subset G$, the conclusion \eqref{eq:modCK} holds for almost every $x \in G$.
\end{proof}

\section{Blowing up cut metrics: proof of Theorem~\ref{thm:blowupcutmetric}}\label{sec:blowup}
In this section, we make the important further step of taking locally uniform sublimits of the rescaled metrics $\tfrac{1}{r}S_{x,r}^*(d_\Sigma)$ considered in Theorem~\ref{thm:modCK}, and examine the structures of the resulting blowup metrics. This is the content of Theorem~\ref{thm:blowupcutmetric}, which we restate here.

\begin{theorem}[Theorem~\ref{thm:blowupcutmetric}]
Let $\Sigma$ be an $\FP_{\loc}$ cut measure on a Carnot group $G$ and $\F \subset \Cut(G)$ 
a collection of cuts such that
\begin{itemize}
    \item $d_\Sigma$ is Lipschitz with respect to $d_G$,
    \item $\F$ is compact,
    \item $\F$ consists of constant normal cuts,
    \item $\F$ is translation and dilation invariant, and
    \item $\F$ contains the $\Sigma$-generic tangents.
\end{itemize} 
Then, for $\vol_G$-a.e.~$x \in G$, every blowup metric $d_{\Sigma,\infty}$ of $d_\Sigma$ at $x$, and every $R \in (0,\infty)$, there exists a cut measure $\Sigma'$ supported on $\F$ such that $\Sigma'(\F) < \infty$ and $d_{\Sigma,\infty} = d_{\Sigma'}$ on $B_R(0) \times B_R(0)$.
\end{theorem}

\begin{proof}
%
Let $x \in G$ be any point such that the conclusion of Theorem~\ref{thm:modCK} holds. First assume $R = 1$. Let $d_{\Sigma,\infty}$ be a blowup metric of $d_\Sigma$ at $x$:
\begin{equation} \label{eq:1}
    d_{\Sigma,\infty} = \lim_{j \to \infty} r_j^{-1}S^*_{x,r_j}(d_{\Sigma})
\end{equation}
for some sequence $(r_j)_j$ decreasing to 0, where the convergence is locally uniform on $G \times G$ (and hence uniform on $B_G \times B_G$). By Theorem~\ref{thm:modCK}, we can find a number $K_x < \infty$ and a sequence of cut measures $\Sigma_j$ supported on $\F$ such that
\begin{align}
    \sup_j \Sigma_j(\F) \leq K_x, \label{eq:2} \\
    \lim_{j \to \infty} \|r_j^{-1}S_{x,r_j}^*(d_\Sigma) - d_{\Sigma_j}\|_{L^1(B_G \times B_G)} = 0. \label{eq:3}
\end{align}

By assumption, the set $\F$ is a compact metrizable space, and hence \eqref{eq:2} shows that $\{\Sigma_j\}_j$ is a weak* precompact subset of $C^0(\F)^*$, the dual space of continuous functions on $\F$. Then, by passing to a subsequence, we may assume that there exists a positive Radon measure $\Sigma'$ on $\F$ with $\Sigma'(\F) \leq K_x$ such that $\Sigma_j \to \Sigma'$ weak*. We will show that $d_{\Sigma_j} \to d_{\Sigma'}$ weakly in $L^1(B_G \times B_G)$.

Let $f \in L^\infty(B_G \times B_G)$. It is easy to check that the assignment $E \mapsto \int_{B_G \times B_G} fd_E \, \dd (\vol_G \times \vol_G)$ is continuous on $\Cut(G)$. Thus, by definition of weak*-convergence and by Fubini Theorem, we get
\begin{align*}
    \int_{B_G \times B_G}fd_{\Sigma'} \,\dd(\vol_G \times \vol_G) &= \int_{B_G \times B_G} f \int_{\F} d_E \: \,\dd \Sigma' \, \dd(\vol_G \times \vol_G) \\
    &= \int_{\F} \int_{B_G \times B_G} f d_E \,\dd(\vol_G \times \vol_G) \: \,\dd \Sigma' \\
    &= \lim_{j \to \infty} \int_{\F} \int_{B_G \times B_G} f d_E \, \dd(\vol_G \times \vol_G) \: \,\dd \Sigma_j \\
    &= \lim_{j \to \infty} \int_{B_G \times B_G} f \int_{\F} d_E \: \,\dd \Sigma_j \,\dd(\vol_G \times \vol_G) \\
    &= \lim_{j \to \infty} \int_{B_G \times B_G}fd_{\Sigma_j} \,\dd(\vol_G \times \vol_G),
\end{align*}
proving $d_{\Sigma_j} \to d_{\Sigma'}$ weakly. Together with \eqref{eq:3}, this shows that
\begin{equation*}
    d_{\Sigma'} = \lim_{j \to \infty} r_j^{-1}S^*_{x,r_j}(d_{\Sigma}),
\end{equation*}
where the convergence is weakly in $L^1(B_G \times B_G)$. By the Dominated Convergence Theorem, uniform converge on $B_G \times B_G$ implies weak convergence in $L^1(B_G \times B_G)$, and thus \eqref{eq:1} implies
\begin{equation*}
    d_{\Sigma,\infty} = \lim_{j \to \infty} r_j^{-1}S^*_{x,r_j}(d_{\Sigma}),
\end{equation*}
where the convergence is weakly. Since the weak topology is Hausdorff, the last two equations imply $d_{\Sigma,\infty} = d_{\Sigma'}$ almost everywhere on $B_G \times B_G$. Since $d_{\Sigma,\infty}$ is Lipschitz with respect to $d_G$, so is $d_{\Sigma'}$. Recall, that in such a setting, we choose the continuos representative of $d_{\Sigma'}$. With this choice, the equality holds everywhere on $B_G\times B_G.$

Now let $R < \infty$ be arbitrary. It is easy to check that the rescaled metric $(y,z) \mapsto \frac{1}{R}d_{\Sigma,\infty}(\delta_R(y),\delta_R(z))$ is another blowup metric of $d_{\Sigma}$ at $x$ (with respect to the sequence of scales $(Rr_j)_j$). Thus, by the above argument applied to this blowup metric, there exists a cut measure $\Sigma_R'$ supported on $\F$ such that $\Sigma_R'(\F) < \infty$ and
\begin{equation} \label{eq:blowupcutmetric1}
    \forall y,z \in B_G, \:\: \frac{1}{R}d_{\Sigma,\infty}(\delta_R(y),\delta_R(z)) = d_{\Sigma_R'}(y,z).
\end{equation}
Now we define a new measure $\Sigma_R''$ on $\Cut(G)$ by $\Sigma_R''(E) := R\Sigma_R'(\delta_{1/R}(E))$. Then since $\F$ is dilation invariant, the cut measure $\Sigma_R''$ is supported on $\F$ and $\Sigma_R''(\F) < \infty$. Furthermore, it is easy to check that $d_{\Sigma_R''}(y,z) = Rd_{\Sigma_R'}(\delta_{1/R}(y),\delta_{1/R}(z))$, and thus when combined with \eqref{eq:blowupcutmetric1} we get that $d_{\Sigma,\infty}(y,z) = d_{\Sigma_R''}(y,z)$ for every $y,z \in B_R(0)$.
\end{proof}

\section{Proof of Theorem~\ref{thm:blowupmetric}}
\label{sec:mainproof}
\subsection{Streamlined Ambrosio-Kleiner-Le Donne}
We begin with a lemma that collects results from \cite{AKLD}, which are essential to us. In that paper, it is shown that, perimeter-almost everywhere, every tangent of a constant normal set has a new invariant direction; an iteration of this procedure generates vertical half-spaces.

Recall that, for a Carnot group $G$ with stratified Lie algebra $\g = \oplus_{i=1}^s V_i$, the step of $G$ is $s$, the rank of $G$ is $m_1 = \dim(V_1)$, and the topological dimension of $G$ is $m_\g = \dim(\g)$. Also recall the notation $\F_k$ from Definition~\ref{def:mathcalF}.


\begin{lemma} \label{lem:AKLD}
Let $G$ be a Carnot group.
\begin{enumerate}
    \item \label{item:AKLD1} Let $\Sigma$ be any $\FP_{\loc}$ cut measure on $G$. Then $\F_{m_1-1}$ contains the $\Sigma$-generic tangents.
    \item \label{item:AKLD2} For each $k \in \N$ with $m_1-1 \leq k \leq m_\g-2$, if $\Sigma$ is a cut measure on $G$ supported on $\F_k$, then $\F_{k+1}$ contains the $\Sigma$-generic tangents.
\end{enumerate}
\end{lemma}

\begin{proof}
Notice that $\F_{m_1-1}$ is simply the collection of constant normal cuts. Then to prove \eqref{item:AKLD1}, since $\Sigma$ is any $\FP_{\loc}$ cut measure, it needs to be shown that the generic tangent of a locally finite-perimeter cut is a constant normal cut. 
This is stated and proved in \cite{FSS03}*{Theorem~3.1} for step 2 Carnot groups, but the proof works for  Carnot groups of arbitrary step, since the step-2 assumption wasn't use up that point in the article. We defer to that article for details.

The proof of \eqref{item:AKLD2} is performed by synthesizing various lemmas and proofs from \cite{AKLD}. Let $k \in \N$ with $m_1-1 \leq k \leq m_\g-2$. We may assume $s \geq 2$, because otherwise $m_1 = m_\g$ and no such $k$ exists. Let $\Sigma$ be a cut measure supported on $\F_k$, and let $E \in \F_k$. It suffices to show that $\F_{k+1}$ contains the generic tangents of $E$. Let $X \in V_1$ be a constant normal for $E$ and $\g' := \spn(\Inv_0(E))$. Observe that $\g'$ is a Lie subalgebra by \cite{AKLD}*{Proposition~4.7(i)}. Then $W := \g' \oplus \R X$ Lie generates $\g$, and thus by \cite{AKLD}*{Proposition~2.17}, there exists $x \in \exp(\g')$ such that $Z := \Ad_x(X) \not\in W$. 
By \cite{AKLD}*{Proposition~4.7(ii)}, $Z \in \Reg(E)$ since $\g' \subset \Inv(E)$ and $X \in \Reg(E)$. Decompose $Z$ as $Z = Z_1 + \dots + Z_s$, with $Z_\ell\in V_\ell$. It must hold that $Z_\ell \not\in \Inv_0(E)$ for some $\ell \geq 2$ since $Z \notin W \supset \spn(\Inv_0(E)) + V_1$. Let $\ell'$ be the largest such $\ell$. Then $Z' := Z_1 + \dots + Z_{\ell'} = Z - (Z_{\ell'+1} + \dots + Z_s) \in \Reg(E)$. 
Then by \cite{AKLD}*{Lemma~5.8}, for $\Per_E$-a.e $x \in G$, and every tangent $L$ of $E$ at $x$, we have that $Z_{\ell'} \in \Inv_0(L)$. 
Since an invariant homogeneous direction is still invariant under a blowup (this follows from Lemma~\ref{lem:distribution}), we thus have $\spn(\Inv_0(E)) \oplus \R Z_{\ell'} \subset \spn(\Inv_0(L))$. This shows $\dim(\spn(\Inv_0(L))) \geq \dim(\spn(\Inv_0(E))) + 1$. Since tangents of constant-normal cuts also have constant normal (this also follows from Lemma~\ref{lem:distribution}), $L \in \F_{k+1}$.
\end{proof}

\begin{remark}
    A rephrasing of the second part of the above lemma is that if $E\in\F_k$, then, for $\Per_E$-a.e $x \in G$, every tangent of $E$ at $x$ belongs to $\F_{k+1}$.
\end{remark}

\subsection{Non-embeddability of non-abelian Carnot groups}

Finally, we conclude by proving Theorem~\ref{thm:blowupmetric}. As explained in the introduction, this implies Theorem~\ref{thm:carnotembed}, which in turn implies our main result: Theorem~\ref{thm:mainthm}.

\begin{proof}[Proof of Theorem~\ref{thm:blowupmetric}]
Let $G$ be a Carnot group and $f: G \to L^1$ Lipschitz. By Proposition~\ref{prop:cut-meas-existence}, there is an $\FP_{\loc}$ cut measure $\Sigma$ on $G$ such that $d_f = d_\Sigma$ $\vol_G \times \vol_G$-almost everywhere. Theorem~\ref{thm:Fkcompact} and Lemma~\ref{lem:AKLD} imply that the hypotheses of Theorem~\ref{thm:blowupcutmetric} are satisfied, and thus, by that theorem, there exists a cut measure $\Sigma_{m_1-1}$ supported on $\F_{m_1-1}$ with $\Sigma_{m_1-1}(\Cut(G)) < \infty$ and $d_{\Sigma_{m_1-1}}$ agrees with a blowup of $d_{\Sigma}$ on $B_G \times B_G$ (which exists by Arzel\`a-Ascoli).
Since the Lipschitz constant does not increase with blowups, $d_{\Sigma_{m_1-1}}$ is also Lipschitz with respect to $d_G$ on $B_G \times B_G$. 
Moreover, since $\Sigma_{m_1-1}(\Cut(G)) < \infty$ and $\Sigma_{m_1-1}$ is supported on $\F_{m_1-1}$, which consists of constant normal cuts, $\Sigma_{m_1-1}$ is an $\FP_{\loc}$ cut measure. 
Indeed, for every $R>0$ and every $E\in\F$ with $\Per_E(B_R(0))>0$, there exists $x_E\in B_R(0)\cap\partial^*E$ by the second part of Lemma~\ref{lem:constantnormalperimeter}, and the first part of that lemma yields
$\Per_E(B_R(0)) \le \Per_E(B_{2R}(x_E)) \le C (2R)^{Q-1}$.
Therefore,
$\int_{\Cut(G)} \Per_E(B_R(0)) \,\dd \Sigma_{m_1-1}(E)
\le C (2R)^{Q-1} \Sigma_{m_1-1}(\Cut(G)) < \infty$.

Repeating the same argument with $d_{\Sigma_{m_1-1}}$ in place of $d_{\Sigma}$, we get that there exists an $\FP_{\loc}$ cut measure $\Sigma_{m_1}$ supported on $\F_{m_1}$ such that $d_{\Sigma_{m_1}}$ agrees with a blowup of $d_{\Sigma_{m_1-1}}$ on $B_G \times B_G$. 
After iterating this procedure up to $m_\g-m_1$ times in total, and after using the stronger form of Theorem~\ref{thm:blowupcutmetric} for the final blowup, we get that there exists a $k$-fold iterated blowup of $\rho$ of $d_\Sigma = d_f$ (with $k \leq m_\g-m_1$) and, for each $R < \infty$, a cut measure $\Sigma^R_{m_\g-1}$ supported on $\F_{m_\g-1}$ such that $d_{\Sigma^R_{m_\g-1}}$ agrees with $\rho$ on $B_R(0) \times B_R(0)$.

Recall that $\F_{m_\g-1}$ is exactly the collection of half-spaces, and in particular $\spn(\cup_{i=2}^s V_i) \subset \Inv(E)$ for every $E \in \F_{m_\g-1}$. By the grading property and basic Lie group theory, it holds that $\exp(\spn(\cup_{i=2}^s V_i)) = [G,G]$. 
Thus, by \cite{BLD}*{Proposition~2.8(1)(2)}, for every $z \in [G,G]$ and $E \in \F_{m_\g-1}$, we have that $\uno_E(xz) = \uno_E(x)$ for $\vol_G$-almost every $x \in G$. Then for every $z \in [G,G]$ and every $R' < \infty$, the definition of $d_E$ and Fubini Theorem imply $d_E(x,yz) = d_E(x,y)$ for $(\Sigma^{R'}_{m_\g-1} \times \vol_G \times \vol_G)$-almost every $(E,x,y) \in \F_{m_\g-1} \times G \times G$.

Now fix $z \in [G,G]$ and let $R' \in (d_G(0,z), \infty)$ be arbitrary. Set $R := R' - d_G(0,z)$, so that $x,yz \in B_{R'}(0)$ whenever $x,y \in B_R(0)$. We define a new continuous function $\rho_z: B_R(0) \times B_R(0) \to \R$ by $\rho_z(x,y) := \rho(x,yz)$. We show next that $\rho_z(x,y) = \rho(x,y)$ for $\vol_G \times \vol_G$-almost every $(x,y) \in B_R(0) \times B_R(0)$ by showing equality as linear functionals on $L^1(B_R(0) \times B_R(0))$. For every $f \in L^1(B_R(0) \times B_R(0))$, Fubini Theorem implies
\begin{align*}
    \rho_z(f) &= \int_{\F_{m_\g-1} \times B_R(0) \times B_R(0)} f(x,y)d_E(x,yz) \, \dd(\Sigma^{R'}_{m_\g-1} \times \vol_G \times \vol_G)(E,x,y) \\
    &= \int_{\F_{m_\g-1} \times B_R(0) \times B_R(0)} f(x,y)d_E(x,y) \, \dd(\Sigma^{R'}_{m_\g-1} \times \vol_G \times \vol_G)(E,x,y) \\
    &= \rho(f).
\end{align*}

\noindent Since both $\rho_z$ and $\rho$ are continuous, this implies $\rho(x,yz) = \rho_z(x,y) = \rho(x,y)$ for every $(x,y) \in B_R(0) \times B_R(0)$. Since $R' < \infty$ was arbitrary and $R \to \infty$ as $R' \to \infty$, this implies $\rho(x,yz) = \rho(x,y)$ for every $x,y \in G$.
\end{proof}

\section{Other spaces non-embeddable into \texorpdfstring{$L^1$}{L1}}
\label{sec:moregroups}
In this final section we prove analogues of Theorem~\ref{thm:mainthm} for other classes of groups and of metric spaces.
The general idea is that if a space quasi-isometrically embeds into $L^1$, then none of its asymptotic cones  can be a nonabelian Carnot group.
Similarly, if a metric space biLipschitz embeds into $L^1$ then none of its tangent spaces can be a nonabelian Carnot group. Both these last statements are immediate consequences of our Theorem~\ref{thm:carnotembed} and Kakutani's representation theorem \cite{BL}*{Corollary~F.4}, as in the proof of Theorem~\ref{thm:mainthm}.

In this section we describe two specific situations where one can exclude quasi-isometric or  biLipschitz embeddings into $L^1$. The first setting is the one of  locally compact groups of polynomial growth, see \cites{Cornulier_delaharpe,Breuillard} for an introduction, terminology, and some results. 
Particular examples of locally compact groups are finitely generated groups equipped with word distances and Lie groups equipped with Riemannian metrics. These last groups are of polynomial growth for example if they are nilpotent.
It is well known (for example from the work of Gromov and Pansu \cites{Gromov81, Pansu83}) that  
finitely generated groups and nilpotent Lie groups are virtually abelian  if and only if they are quasi-isometric to some Euclidean space. We shall show that this last property is necessary and sufficient for quasi-isometric embeddability into $L^1$.

\begin{corollary}
A locally compact group of polynomial growth embeds quasi-isometrically into $L^1$ if and only if it is
quasi-isometric to an Euclidean space. 
\end{corollary}

\begin{proof}
As Euclidean spaces biLipschitz embed into $L^1$,
one implication is obvious. For the other implication, let $G$ be 
a locally compact group of polynomial growth.
By Breuillard's study of locally compact groups of polynomial growth \cite{Breuillard}*{Theorem~1.2 and Lemma~3.110} we have that
$G$ is quasi-isometric to a connected simply connected nilpotent Riemannian Lie group $N$. 
Hence, 
if $G$ admits a quasi-isometric embedding into $L^1$, then so does  $N$. 
From Theorem~\ref{thm:mainthm} we infer that $N$ is abelian, hence an Euclidean space.
\end{proof}

The proof of the last corollary actually gives another corollary, since one can substitute Breuillard's result with  \cite{Cowling_et_al}*{Corollary~4.33}. This result states, that a metric space is 
quasi-isometric to some connected simply connected nilpotent Riemannian Lie group under the assumption that
it is
  boundedly compact, connected, quasigeodesic, and homogeneous. 
For these notions we refer to \cite{Cowling_et_al}. 
\begin{corollary}

Let $X$ be a metric space that is 
boundedly compact, connected, quasigeodesic, and homogeneous. Then $X$ embeds quasi-isometrically into $L^1$ if and only if it is
quasi-isometric to an Euclidean space. 

\end{corollary}


The second setting where Carnot groups appear naturally is subRiemannian geometry. We refer to \cite{ABB} for an introduction to subRiemannian manifolds. Like the Heisenberg group, subRiemannian manifolds have provided several examples of nonembeddability results, unless they are Riemannian. We extend these results with $L^1$ target.

\begin{corollary}\label{cor:subRiemMan}
An equiregular subRiemannian manifold that biLipschitz embeds into $L^1$ is Riemannian.
\end{corollary}
\begin{proof}
By Mitchell's theorem \cite{Bellaiche}, an equiregular subRiemannian manifold $M$ admits at every point a tangent cone that is a Carnot group, which is Euclidean only at those points where $M$ is Riemannian.
If $M$ 
biLipschitz embeds into $L^1$, then as a corollary to Kakutani's representation theorem \cite{BL}*{Corollary~F.4}, the biLipschitz embedding
induces a biLipschitz embedding of each of the tangent cones of $M$ into $L^1$. Theorem~\ref{thm:carnotembed} implies that 
each of these tangent cones is an Euclidean space, and thus the subRiemannian structure of $M$ is Riemannian. 
\end{proof}

\begin{remark}
    The requirement that the subRiemannian manifold is equiregular cannot be dropped in Corollary~\ref{cor:subRiemMan}. Without this assumption, our argument shows that the manifold is almost Riemannian (see~\cite{ABB}). One cannot conclude that the manifold is Riemannian because 
     it has been shown that the Grushin plane
     biLipschitz embeds into the Euclidean 3-space, and thus into $L^1$, see~\cites{seo,wu}.
\end{remark}

\bibliographystyle{amsalpha}

\begin{bibdiv}
\begin{biblist}

\bib{Ambrosio}{article}{
   author={Ambrosio, Luigi},
   title={Some fine properties of sets of finite perimeter in Ahlfors regular metric measure spaces},
   journal={Adv. Math.},
   volume={159},
   date={2001},
   number={1},
   pages={51--67},
}

\bib{AKLD}{article}{
    AUTHOR = {Ambrosio, Luigi},
    author = {Kleiner, Bruce},
    author = {Le Donne, Enrico},
     TITLE = {Rectifiability of sets of finite perimeter in {C}arnot groups: existence of a tangent hyperplane},
   JOURNAL = {J. Geom. Anal.},
    VOLUME = {19},
      YEAR = {2009},
    NUMBER = {3},
     PAGES = {509--540}
}

\bib{AIR}{article}{
   author={Andoni, Alexandr},
   author={Indyk, Piotr},
   author={Razenshteyn, Ilya},
   title={Approximate nearest neighbor search in high dimensions},
   conference={
      title={Proceedings of the International Congress of
      Mathematicians---Rio de Janeiro 2018. Vol. IV. Invited lectures},
   },
   book={
      publisher={World Sci. Publ., Hackensack, NJ},
   },
   date={2018},
   pages={3287--3318},
}

\bib{ABB}{book}{
	author = {Agrachev, Andrei},
	author = {Barilari, Davide},
	author = {Boscain, Ugo},
	publisher = {Cambridge University Press, Cambridge},
	series = {Cambridge Studies in Advanced Mathematics},
	title = {A comprehensive introduction to sub-{R}iemannian geometry},
	volume = {181},
	year = {2020}
}

\bib{BMS}{article}{
   author={Baudier, Florent P.},
   author={Motakis, Pavlos},
   author={Schlumprecht, Thomas},
   author={Zs\'{a}k, András},
   title={On the bi-Lipschitz geometry of lamplighter graphs},
   journal={Discrete Comput. Geom.},
   volume={66},
   date={2021},
   number={1},
   pages={203--235},
   issn={0179-5376},
}

\bib{Bellaiche}{article}{
	author = {Bella\"{\i}che, Andr{\'e}},
	booktitle = {Sub-{R}iemannian geometry},
	pages = {1--78},
	publisher = {Birkh{\"a}user, Basel},
	series = {Progr. Math.},
	title = {The tangent space in sub-{R}iemannian geometry},
	volume = {144},
	year = {1996}
}

\bib{BLD}{article}{
    AUTHOR = {Bellettini, Constante},
    author = {Le Donne, Enrico},
     TITLE = {Sets with constant normal in {C}arnot groups: properties and examples},
   JOURNAL = {Comment. Math. Helv.},
    VOLUME = {96},
      YEAR = {2021},
    NUMBER = {1},
     PAGES = {149--198},
}

\bib{BL}{book}{
   author={Benyamini, Yoav},
   author={Lindenstrauss, Joram},
   title={Geometric nonlinear functional analysis. Vol. 1},
   series={American Mathematical Society Colloquium Publications},
   volume={48},
   publisher={American Mathematical Society, Providence, RI},
   date={2000},
   pages={xii+488},
}

\bib{Breuillard}{article}{
	author = {Breuillard, Emmanuel},
	journal = {Groups Geom. Dyn.},
	number = {3},
	pages = {669--732},
	title = {Geometry of locally compact groups of polynomial growth and shape of large balls},
	volume = {8},
	year = {2014}
}

\bib{CKRNP}{article}{
   author={Cheeger, Jeff},
   author={Kleiner, Bruce},
   title={Differentiability of Lipschitz maps from metric measure spaces to Banach spaces with the Radon-Nikod\'{y}m property},
   journal={Geom. Funct. Anal.},
   volume={19},
   date={2009},
   number={4},
   pages={1017--1028},
}

\bib{CKL1}{article}{
    author={Cheeger, Jeff},
    author={Kleiner, Bruce},
     TITLE = {Differentiating maps into {$L^1$}, and the geometry of {BV} functions},
   JOURNAL = {Ann. of Math.},
    VOLUME = {171},
      YEAR = {2010},
    NUMBER = {2},
     PAGES = {1347--1385}
}

\bib{CKmonotone}{article}{
   author={Cheeger, Jeff},
   author={Kleiner, Bruce},
   title={Metric differentiation, monotonicity and maps to $L^1$},
   journal={Invent. Math.},
   volume={182},
   date={2010},
   number={2},
   pages={335--370},
   issn={0020-9910}
}

\bib{Cornulier_delaharpe}{article}{
	author = {Cornulier, Yves and de la Harpe, Pierre},
	publisher = {European Mathematical Society (EMS), Z{\"u}rich},
	series = {EMS Tracts in Mathematics},
	title = {Metric geometry of locally compact groups},
	volume = {25},
	year = {2016}
}

\bib{Cowling_et_al}{article}{
	author = {Cowling, Michael G.},
	author = {Kivioja, Ville},
	author = {Le Donne, Enrico},
	author = {Nicolussi Golo, Sebastiano},
	author = {Ottazzi, Alessandro},
	journal = {ArXiv e-prints, version 4},
	title = {From homogeneous metric spaces to Lie groups},
	year = {2017}
}

\bib{Drutu_Kapovich}{book}{
	author = {Dru\c{t}u, Cornelia},
	author = {Kapovich, Michael},
	publisher = {American Mathematical Society, Providence, RI},
	series = {American Mathematical Society Colloquium Publications},
	title = {Geometric group theory},
	volume = {63},
	year = {2018}
}

\bib{franchiPI}{incollection}{
    AUTHOR = {Franchi, Bruno},
     TITLE = {B{V} spaces and rectifiability for {C}arnot-{C}arath\'{e}odory
              metrics: an introduction},
 BOOKTITLE = {N{AFSA} 7---{N}onlinear analysis, function spaces and
              applications. {V}ol. 7},
     PAGES = {72--132},
 PUBLISHER = {Czech. Acad. Sci., Prague},
      YEAR = {2003},
}

\bib{FSS03}{article}{
    AUTHOR = {Franchi, Bruno},
    author = {Serapioni, Raul},
    author = {Serra Cassano, Francesco},
     TITLE = {On the structure of finite perimeter sets in step 2 {C}arnot groups},
   JOURNAL = {J. Geom. Anal.},
    VOLUME = {13},
      YEAR = {2003},
    NUMBER = {3},
     PAGES = {421--466},
}
		
\bib{GN}{article}{
    AUTHOR = {Garofalo, Nicola},
    author = {Nhieu, Duy-Minh},
     TITLE = {Isoperimetric and {S}obolev inequalities for
              {C}arnot-{C}arath\'{e}odory spaces and the existence of minimal surfaces},
   JOURNAL = {Comm. Pure Appl. Math.},
    VOLUME = {49},
      YEAR = {1996},
    NUMBER = {10}
}

\bib{Gromov81}{article}{
	author = {Gromov, Mikhael},
	journal = {Inst. Hautes {\'E}tudes Sci. Publ. Math.},
	number = {53},
	pages = {53--73},
	title = {Groups of polynomial growth and expanding maps},
	year = {1981}
}

\bib{Gromov}{article}{
   author = {Gromov, Mikhael},
   title = {Asymptotic invariants of infinite groups},
   conference={
      title={Geometric group theory, Vol. 2},
      address={Sussex},
      date={1991},
   },
   book={
      series={London Math. Soc. Lecture Note Ser.},
      volume={182},
      publisher={Cambridge Univ. Press, Cambridge},
   },
   date={1993},
   pages={1--295},
}

\bib{GNRS}{article}{
   author={Gupta, Anupam},
   author={Newman, Ilan},
   author={Rabinovich, Yuri},
   author={Sinclair, Alistair},
   title={Cuts, trees and $l_1$-embeddings of graphs},
   journal={Combinatorica},
   volume={24},
   date={2004},
   number={2},
   pages={233--269},
}

\bib{Heinonen}{book}{
   author={Heinonen, Juha},
   title={Lectures on analysis on metric spaces},
   series={Universitext},
   publisher={Springer-Verlag, New York},
   date={2001},
   pages={x+140},
   isbn={0-387-95104-0},
}
		

\bib{hormander}{book}{
    AUTHOR = {H\"{o}rmander, Lars},
     TITLE = {The analysis of linear partial differential operators. {I}},
    SERIES = {Classics in Mathematics},
      NOTE = {Distribution theory and Fourier analysis,
              Reprint of the second (1990) edition [Springer, Berlin;
              MR1065993 (91m:35001a)]},
 PUBLISHER = {Springer-Verlag, Berlin},
      YEAR = {2003},
     PAGES = {x+440},
      ISBN = {3-540-00662-1},
}

\bib{PR}{article}{
   author={Indyk, Piotr},
   author={Motwani, Rajeev},
   title={Approximate nearest neighbors: towards removing the curse of
   dimensionality},
   conference={
      title={STOC '98 (Dallas, TX)},
   },
   book={
      publisher={ACM, New York},
   },
   date={1999},
   pages={604--613},
}

\bib{KhotNaor}{article}{
   author={Khot, Subhash},
   author={Naor, Assaf},
   title={Nonembeddability theorems via Fourier analysis},
   journal={Math. Ann.},
   volume={334},
   date={2006},
   number={4},
   pages={821--852},
}

\bib{LD}{article}{
    author = {Le Donne, Enrico},
    title = {A primer on Carnot groups: homogenous groups, Carnot-Carath\'eodory spaces, and regularity of their isometries},
    journal = {Anal. Geom. Metr. Spaces},
    volume = {5},
    year = {2017},
    pages = {116--137},
}

\bib{LeeNaor}{article}{
    author = {Lee, James R.},
    author = {Naor, Assaf},
    title = {Lp metrics on the Heisenberg group and the Goemans-Linial conjecture},
    journal = {2006 47th Annual IEEE Symposium on Foundations of Computer Science (FOCS’06)},
    date = {2006},
    pages =  {99--108}
}

\bib{LNP}{article}{
   author={Lee, James R.},
   author={Naor, Assaf},
   author={Peres, Yuval},
   title={Trees and Markov convexity},
   journal={Geom. Funct. Anal.},
   volume={18},
   date={2009},
   number={5},
   pages={1609--1659},
}

\bib{LLR}{article}{
   author={Linial, Nathan},
   author={London, Eran},
   author={Rabinovich, Yuri},
   title={The geometry of graphs and some of its algorithmic applications},
   journal={Combinatorica},
   volume={15},
   date={1995},
   number={2},
   pages={215--245},
}

\bib{NaorPeres1}{article}{
   author={Naor, Assaf},
   author={Peres, Yuval},
   title={Embeddings of discrete groups and the speed of random walks},
   journal={Int. Math. Res. Not. IMRN},
   date={2008},
   pages={Art. ID rnn 076, 34},
   issn={1073-7928},
}

\bib{NaorPeres2}{article}{
   author={Naor, Assaf},
   author={Peres, Yuval},
   title={$L_p$ compression, traveling salesmen, and stable walks},
   journal={Duke Math. J.},
   volume={157},
   date={2011},
   number={1},
   pages={53--108},
   issn={0012-7094},
}

\bib{NaorYoung}{article}{
   author={Naor, Assaf},
   author={Young, Robert},
   title={Vertical perimeter versus horizontal perimeter},
   journal={Ann. of Math. (2)},
   volume={188},
   date={2018},
   number={1},
   pages={171--279},
}

\bib{Ostrovskiibook}{book}{
   author={Ostrovskii, Mikhail I.},
   title={Metric embeddings},
   series={De Gruyter Studies in Mathematics},
   volume={49},
   note={Bilipschitz and coarse embeddings into Banach spaces},
   publisher={De Gruyter, Berlin},
   date={2013},
   pages={xii+372},
}

\bib{Ostrovskii}{article}{
   author={Ostrovskii, Mikhail},
   title={Metric characterizations of superreflexivity in terms of word
   hyperbolic groups and finite graphs},
   journal={Anal. Geom. Metr. Spaces},
   volume={2},
   date={2014},
   number={1},
   pages={154--168},
}

\bib{Pansu83}{article}{
	author = {Pansu, Pierre},
	journal = {Ergodic Theory Dynam. Systems},
	number = {3},
	pages = {415--445},
	title = {Croissance des boules et des g{\'e}od{\'e}siques ferm{\'e}es dans les nilvari{\'e}t{\'e}s},
	volume = {3},
	year = {1983}
}

\bib{Pansu}{article}{
   author={Pansu, Pierre},
   title={M\'{e}triques de Carnot-Carath\'{e}odory et quasiisom\'{e}tries des espaces
   sym\'{e}triques de rang un},
   language={French, with English summary},
   journal={Ann. of Math. (2)},
   volume={129},
   date={1989},
   number={1},
   pages={1--60},
}

\bib{Pi}{book}{
    AUTHOR = {Pisier, Gilles},
     TITLE = {Martingales in {B}anach spaces},
    SERIES = {Cambridge Studies in Advanced Mathematics},
    VOLUME = {155},
 PUBLISHER = {Cambridge University Press, Cambridge},
      YEAR = {2016},
     PAGES = {xxviii+561},
      ISBN = {978-1-107-13724-0},
}

\bib{seo}{article}{
	author = {Seo, Jeehyeon},
	journal = {Math. Res. Lett.},
	number = {6},
	pages = {1179--1202},
	title = {A characterization of bi-{L}ipschitz embeddable metric spaces in terms of local bi-{L}ipschitz embeddability},
	volume = {18},
	year = {2011}
}

\bib{wu}{article}{
	author = {Wu, Jang-Mei},
	journal = {Ann. Sc. Norm. Super. Pisa Cl. Sci. (5)},
	number = {2},
	pages = {633--644},
	title = {Bilipschitz embedding of {G}rushin plane in {$\mathbb{R}^3$}},
	volume = {14},
	year = {2015}
}

\bib{Yu}{article}{
   author={Yu, Guoliang},
   title={The coarse Baum-Connes conjecture for spaces which admit a uniform
   embedding into Hilbert space},
   journal={Invent. Math.},
   volume={139},
   date={2000},
   number={1},
   pages={201--240},
}

\end{biblist}
\end{bibdiv}

\end{document}